\documentclass[reqno,11pt]{marticle}
\usepackage{amsmath,mamsthm,amssymb} % amsthm
\usepackage{amsfonts}

\newtheoremstyle{theorem}%name
{5pt} % space above
{5pt} % space below
{\sl} % bofy font
{\parindent} % ident - empty=no indent, \parindent= paragraph indent
{\bf} % thm head font
{. } % punctuation after thm head
{ } % space after thm head: `` ``=normal \newline=linebreak
{} % thm head specification
\theoremstyle{theorem}
\newtheorem{theorem}{Theorem}
\theoremstyle{theorem}
\newtheorem{corollary}[theorem]{Corollary}
\newtheorem{proposition}[theorem]{Proposition}

\newtheoremstyle{defi}%name
{5pt} % space above
{5pt} % space below
{\rm} % bofy font
{\parindent} % ident - empty=no indent, \parindent= paragraph indent
{\bf} % thm head font
{. } % punctuation after thm head
{ } % space after thm head: `` ``=normal \newline=linebreak
{} % thm head specification
\theoremstyle{defi}
\newtheorem{definition}[theorem]{Definition}
\def\proofname{\indent {\sl Proof.}}

\theoremstyle{defi}
\newtheorem{remark}[theorem]{Remark}

\newtheorem{example}[theorem]{Example}

\begin{document}
\centerline{\hfill}
\centerline{\hfill}
\centerline{\hfill}
%\centerline{\hfill}
%\centerline{\hfill}                        
			%title
\begin{center}
\bf
INTEGRAL OPERATORS WITH INFINITELY SMOOTH BI-CARLEMAN KERNELS OF MERCER TYPE
\end{center}
%\centerline{\hfill}
%\centerline{\hfill}
\vskip.5cm
                         %author
\centerline{Igor M. Novitskii}
                         %address
\begin{center}
Khabarovsk Division\\
Institute of Applied Mathematics \\
Far-Eastern Branch of the Russian Academy of Sciences \\
54, Dzerzhinskiy Street, Khabarovsk 680 000, RUSSIA\\
e-mail: novim@iam.khv.ru
\end{center}
                                %abstract
\noindent{\bf Abstract:} 
With the aim of applications to solving general integral equations,
we introduce and study in this paper a special class of bi-Carleman
kernels on $\mathbb{R}\times\mathbb{R}$, called $K^\infty$ kernels of Mercer
type, whose property of being infinitely smooth is stable under passage to
certain left and right multiples of their associated integral operators. An
expansion theorem in absolutely and uniformly convergent bilinear series
concerning kernels of this class is proved extending to a general non-Hermitian
setting both Mercer's and Kadota's Expansion Theorems for positive definite
kernels. Another theorem proved in this paper identifies families of those
bounded operators on a separable Hilbert space $\mathcal{H}$ that can be
simultaneously transformed by the same unitary equivalence transformation into
bi-Carleman integral operators on $L^2(\mathbb{R})$, whose kernels are
$K^\infty$ kernels of Mercer type; its singleton version implies in particular
that any bi-integral operator is unitarily equivalent to an integral operator
with such a kernel.
\vskip\baselineskip

\section{Introduction and the Main Results}
In the theory of general linear integral equations in $L^2$ spaces,
the equations with bounded infinitely differentiable bi-Carleman kernels
(termed $K^\infty$ kernels) should and do lend themselves well to solution
by approximation and variational methods. The question of whether a second-kind
integral equation with arbitrary kernel can be reduced to an equivalent one
with a $K^\infty$ kernel was positively answered using a unitary-reduction
method by the author \cite{nov:IJPAM2}. In the present paper, we enrich the
$K^\infty$ kernels by imposing an extra condition (Definition~\ref{MerKernel})
that guarantees the kernels and their partial and strong derivatives to be
expandable in absolutely and uniformly convergent bilinear series which may
also be relevant when solving integral equations. We then show that that
condition can always be achieved by means of a unitary reduction which
involves no loss of generality in the study of integral equations of the
second kind.
\par
Before we can write down our main results, we need to fix the terminology and
notation and to give some definitions and preliminary material. Throughout
this paper, $\mathcal{H}$ is a complex, separable, infinite-dimensional
Hilbert space with the inner product
$\langle\cdot,\cdot\rangle_{\mathcal{H}}$ and the norm
$\left\|\cdot\right\|_{\mathcal{H}}$. If $L\subset\mathcal{H}$, we write
$\overline{L}$ for the norm closure of $L$ in $\mathcal{H}$, and
$\mathrm{Span}(L)$ for the norm closure of the set of all linear combinations
of elements of $L$. Let $\mathfrak{R}({\mathcal{H}})$ be the Banach algebra of
all bounded linear operators acting on ${\mathcal{H}}$. For an operator $A$ of
$\mathfrak{R}(\mathcal{H})$, $A^*$ stands for the adjoint to $A$ with respect
to $\langle\cdot,\cdot\rangle_{\mathcal{H}}$,
$\mathrm{Ran\,}A=\left\{Af\mid f\in\mathcal{H}\right\}$ for the range of $A$.
An operator $A\in\mathfrak{R}(\mathcal{H})$ is said to be \textit{invertible}
if it has an inverse which is also in $\mathfrak{R}(\mathcal{H})$, that is, if
there is an operator $B\in\mathfrak{R}(\mathcal{H})$ for which
$BA=AB=I_{{\mathcal{H}}}$ where $I_{{\mathcal{H}}}$ is the identity operator
on ${\mathcal{H}}$; $B$ is denoted by $A^{-1}$. An operator
$P\in\mathfrak{R}(\mathcal{H})$ is called \textit{positive} and denoted by
$P\ge0$ if $\langle Px,x\rangle_{\mathcal{H}}\ge 0$ for all $x\in\mathcal{H}$.
Also recall that a normal operator $A\in\mathfrak{R}(\mathcal{H})$
($AA^*=A^*A$) is called \textit{diagonal}(\textit{izable}) if $\mathcal{H}$
admits an orthonormal basis consisting of eigenvectors of $A$.
If $T\in\mathfrak{R}({\mathcal{H}})$, define the operator families
$\mathcal{M}(T)$ and $\mathcal{M}^+(T)$ by
\begin{gather}
\mathcal{M}(T)=\left(T\mathfrak{R}({\mathcal{H}})\cup
T^*\mathfrak{R}({\mathcal{H}})\right)\cap
\left(\mathfrak{R}({\mathcal{H}})T\cup\mathfrak{R}({\mathcal{H}})T^*\right),
\label{MT}\\
\mathcal{M}^{+}(T)=\{P\in\mathcal{M}(T)\mid P\ge0\}\label{Pc},
\end{gather}
where
$S\mathfrak{R}({\mathcal{H}})=\{SV\mid V\in \mathfrak{R}({\mathcal{H}})\}$,
$\mathfrak{R}({\mathcal{H}})S=\{VS\mid V\in \mathfrak{R}({\mathcal{H}})\}$.
The following are some simple remarks about families just defined.
\begin{remark}\label{remmt}
First note from \eqref{MT} that $\mathcal{M}(S)=\mathcal{M}(S^*)$, and that
the set $\mathcal{M}(S)$ is alternatively defined as
\begin{gather*}
\begin{split}
\mathcal{M}(S)=&
    \left(S\mathfrak{R}(\mathcal{H})\cap\mathfrak{R}(\mathcal{H})S\right)\cup
    \left(S^*\mathfrak{R}(\mathcal{H})\cap\mathfrak{R}(\mathcal{H})S^*\right)
           \\
&\cup\left(S\mathfrak{R}(\mathcal{H})\cap\mathfrak{R}(\mathcal{H})S^*\right)
\cup\left(S^*\mathfrak{R}(\mathcal{H})\cap\mathfrak{R}(\mathcal{H})S\right),
\end{split}
\end{gather*}
which is the same as saying that an operator $A$ belongs to $\mathcal{M}(S)$
if and only if there exist two operators $M$, $N\in\mathfrak{R}(\mathcal{H})$
such that at least one of the four relations
\begin{equation}\label{relations}
\begin{split}
A=SM&=NS,\quad A=S^*M=NS^*, \\&
A=SM=NS^*,\quad \text{or}\ A=S^*N=MS,
\end{split}
\end{equation}
holds. The family $\mathcal{M}^{+}(S)$ can then be characterized as the set of
all positive operators $P\in\mathfrak{R}(\mathcal{H})$ that are expressible as
$P=SB$, or as $P=BS$, where  $B\in\mathfrak{R}(\mathcal{H})$, that is to say,
$\mathcal{M}^{+}(S)=
\left\{P\in S\mathfrak{R}(\mathcal{H})\cup\mathfrak{R}(\mathcal{H})S\mid P
\ge 0\right\}$. If, finally, $S\in\mathcal{M}(T)$, then, by \eqref{relations},
there are four operators $K$, $L$, $Q$, $R\in\mathfrak{R}(\mathcal{H})$ such
that
$
\mathcal{M}(S)=\left(TK\mathfrak{R}(\mathcal{H})\cup T^*L\mathfrak{R}(\mathcal{H})\right)\cap
\left(\mathfrak{R}(\mathcal{H})QT^*\cup \mathfrak{R}(\mathcal{H})RT)\right)
$,
whence it follows via \eqref{MT} that $\mathcal{M}(S)\subseteq\mathcal{M}(T)$,
and, consequently, $\mathcal{M}^{+}(S)\subseteq\mathcal{M}^{+}(T)$.
\end{remark}
\begin{definition}\label{mfac}
A factorization of an operator $T\in\mathfrak{R}({\mathcal{H}})$ into the
product
\begin{equation}\label{f1}
       T = WV^*,      
\end{equation}
provided that the operators $V$, $W\in\mathfrak{R}({\mathcal{H}})$ are subject
to the provisos
\begin{equation}\label{f2}
VV^*\in\mathcal{M}^{+}(T),\quad WW^*\in\mathcal{M}^{+}(T),
\end{equation}
is called an $\mathcal{M}$ \textit{factorization for $T$}.
\end{definition}
\begin{example}\label{TWV}
An ingenuous example of an $\mathcal{M}$ factorization for any
$T\in\mathfrak{R}(\mathcal{H})$ is easy to come by. Just put $W=UP$ and $V=P$,
where $P$ is the positive square root of $|T|=(T^*T)^{\frac1{2}}$ and $U$ is
the partially isometric factor in the polar decomposition $T=U|T|$; since
$T=WV^*$, $WW^*=U|T|U^*=TU^*=UT^*\in\mathcal{M}^{+}(T)$, and
$VV^*=|T|=T^*U=U^*T\in\mathcal{M}^{+}(T)$, it follows that the requirements
\eqref{f1}, \eqref{f2} are satisfied.
\end{example}
The set of \textit{regular values} for the operator
$T\in\mathfrak{R}(\mathcal{H})$, denoted by $\Pi(T)$, is the set of complex
numbers $\lambda$ such that the operator $I_{{\mathcal{H}}}-\lambda T$ is
invertible, that is, it has an inverse
$R_\lambda(T)=\left(I_{{\mathcal{H}}}-\lambda T\right)^{-1}
\in\mathfrak{R}({\mathcal{H}})$ that satisfies
\begin{equation}\label{eqress}
\left(I_{{\mathcal{H}}}-\lambda T\right)R_\lambda(T)=R_\lambda(T)
\left(I_{\mathcal{H}}-\lambda T\right)=I_{\mathcal{H}}.
\end{equation}
The operator
\begin{equation}\label{eqresf}
T_\lambda:= TR_\lambda(T)\ (=R_\lambda(T)T)
\end{equation}
of $\mathcal{M}(T)$ is then referred to as the \textit{Fredholm resolvent} of
$T$ at $\lambda$. Remark that if $\lambda$ is a regular value for $T$, then,
for each fixed $g\in\mathcal{H}$, the (unique) solution $f\in\mathcal{H}$
to the second-kind equation $f-\lambda Tf=g$ may be written as
\begin{equation}\label{equnsol}
f=g+\lambda T_\lambda g.
\end{equation}
\par
Let $\mathbb{R}$ be the real line $(-\infty,+\infty)$ equipped with the Lebesgue
measure, and let $L^2=L^2(\mathbb{R})$ be the Hilbert space of (equivalence classes
of) measurable complex-valued functions on $\mathbb{R}$ equipped with the inner
product $\langle f,g\rangle_{L^2}=\int_{\mathbb{R}}f(s)\overline{g(s)}\,ds$ and the norm
$\|f\|_{L^2}=\langle f,f\rangle^{\frac{1}2}$. A linear operator $T\colon L^2\to L^2$
is \textit{integral} if there is a complex-valued measurable function
$\boldsymbol{T}$ (\textit{kernel}) on the Cartesian product
$\mathbb{R}^2=\mathbb{R}\times\mathbb{R}$ such that
\begin{equation*}
               (Tf)(s)=\int_{\mathbb{R}} \boldsymbol{T}(s,t)f(t)\,dt
\end{equation*}
for every $f\in L^2$ and almost every $s\in\mathbb{R}$. Recall
\cite[Theorem 3.10]{Halmos:Sun} that integral operators are bounded, and need
not be compact. A kernel $\boldsymbol{T}$ on $\mathbb{R}^2$ is said to be
\textit{Carleman} if $\boldsymbol{T}(s,\cdot)\in L^2$ for almost every fixed
$s$ in $\mathbb{R}$. To each Carleman kernel $\boldsymbol{T}$ there corresponds
a \textit{Carleman function} $\boldsymbol{t}\colon \mathbb{R}\to L^2$ defined by
$\boldsymbol{t}(s)=\overline{\boldsymbol{T}(s,\cdot)}$ for all $s$ in
$\mathbb{R}$ for which $\boldsymbol{T}(s,\cdot)\in L^2$. The Carleman kernel
$\boldsymbol{T}$ is called \textit{bi-Carleman} in case its conjugate
transpose kernel $\boldsymbol{T}^{\boldsymbol{\prime}}$
($\boldsymbol{T}^{\boldsymbol{\prime}}(s,t)=\overline{\boldsymbol{T}(t,s)}$)
is also a Carleman kernel. Associated with the conjugate transpose
$\boldsymbol{T}^{\boldsymbol{\prime}}$ of every bi-Carleman kernel
$\boldsymbol{T}$ there is therefore a Carleman function
$\boldsymbol{t}^{\boldsymbol{\prime}}\colon \mathbb{R}\to L^2$ defined by
$\boldsymbol{t}^{\boldsymbol{\prime}}(s)
=\overline{\boldsymbol{T}^{\boldsymbol{\prime}}(s,\cdot)}
\left(=\boldsymbol{T}(\cdot,s)\right)$ for all $s$ in $\mathbb{R}$
for which $\boldsymbol{T}^{\boldsymbol{\prime}}(s,\cdot)\in L^2$.
With each bi-Carleman kernel $\boldsymbol{T}$, we therefore associate the pair
of Carleman functions $\boldsymbol{t}$,
$\boldsymbol{t}^{\boldsymbol{\prime}}\colon \mathbb{R}\to L^2$, both defined, via
$\boldsymbol{T}$, as above. An integral operator whose kernel is Carleman
(resp., bi-Carleman) is referred to as the \textit{Carleman}
(resp., \textit{bi-Carleman}) operator. The integral operator $T$ is called
\textit{bi-integral} if its adjoint $T^*$ is also an integral operator;
in that case if $\boldsymbol{T}^{\boldsymbol{\ast}}$ is the kernel of $T^*$
then, in the above notation,
$\boldsymbol{T}^{\boldsymbol{\ast}}(s,t)=\boldsymbol{T}^{\boldsymbol{\prime}}(s,t)$
for almost all $(s,t)\in\mathbb{R}^2$ (see, e.g., \cite[Theorem 7.5]{Halmos:Sun}).
A bi-Carleman operator is always a bi-integral operator, but not conversely.
Henceforth in this paper, kernels and Carleman functions will always be denoted
by bold-face uppercase and bold-face lowercase letters, respectively.
\par
\begin{remark}\label{rem1} From the viewpoint of the foundations of integral
equation theory, the bad news about Fredholm resolvents is that the property
of being an integral operator is not shared in general by Fredholm resolvents
of integral operators; an example of an integral operator whose Fredholm
resolvent at any non-zero regular value is not an integral operator can be
found in \cite{Kor:nonint1}, or in \cite[Section~5, Theorem~8]{Kor:alg}. But,
fortunately, this phenomenon can never be extended to Carleman operators due
to an algebraic property of these operators, a property which is the content
of the so called ``Right-Multiplication Lemma'' (see \cite[Theorem 11.6]{Halmos:Sun}
or \cite[Corollary IV.2.8]{Kor:book1}):
\begin{proposition}
\label{rimlt}
Let $T$ be a Carleman operator, let $\boldsymbol{t}$ be the Carleman
function associated with the inducing Carleman kernel of $T$, and let
$A\in\mathfrak{R}\left(L^2\right)$ be arbitrary. Then the product operator $TA$
is also a Carleman operator, and the composition function $A^*(\boldsymbol{t}(\cdot))\colon \mathbb{R}\to L^2$
is the Carleman function corresponding with its kernel.
\end{proposition}
Thus, by \eqref{eqresf}, the Fredholm resolvent of a Carleman (resp.,
bi-Carleman) operator at its every regular value is always a Carleman (resp.,
bi-Carleman) operator. Furthermore, if the operator $S$ is bi-Carleman, then
it follows from the above proposition and \eqref{relations} that the
corresponding family $\mathcal{M}(S)$ must consist of bi-Carleman operators 
only. In the general theory of integral equations of the second kind in $L^2$, that is,
equations of the form
\begin{equation}\label{skequ}
f(s)-\lambda\int_{\mathbb{R}} \boldsymbol{T}(s,t)f(t)\,dt =g(s)\quad
\text{for almost all $s\in\mathbb{R}$},
\end{equation}
it is customary to call the kernel $\boldsymbol{T}_\lambda$ (when it exists)
of the Fredholm resolvent $T_\lambda$ of a bi-integral operator $T$, induced
on $L^2$ by the kernel $\boldsymbol{T}$, a \textit{resolvent kernel for
$\boldsymbol{T}$ at $\lambda$}. Once the resolvent kernel
$\boldsymbol{T}_\lambda$ comes to be known, one can express the $L^2$ solution
$f$ to equation \eqref{skequ} in a direct fashion as
\begin{equation*}\label{solution}
            f(s)=g(s)+\lambda\int_{\mathbb{R}} \boldsymbol{T}_\lambda(s,t)g(t)\,dt,
\end{equation*}
regardless of the particular choice of the function $g$ of $L^2$ (cf.
\eqref{equnsol}). Therefore, in the case when the kernel $\boldsymbol{T}$ is
Carleman and $\lambda\in\Pi(T)$ is arbitrary, the problem of solving equation
\eqref{skequ} may be reduced to the problem of explicitly constructing (in
terms of $\boldsymbol{T}$) the resolvent kernel $\boldsymbol{T}_\lambda$ which
is a priori known to exist. For a precise formulation of this latter problem
(not solved here) and for comments to the solution of some of its special cases
we refer to the works by Korotkov \cite{Kor:problems}, \cite{Kor:alg} (in both
the references, see Problem~4 in \S 5). Here we only recall that as long as a
measurable kernel $\boldsymbol{T}$ of \eqref{skequ} is bi-Carleman but
otherwise unrestricted, there seems to be as yet no analytic machinery for
explicitly constructing its resolvent kernel $\boldsymbol{T}_\lambda$ at all
$\lambda\in\Pi(T)$.
With the objective of making problems like this more tractable,
we propose in the present paper to treat some analytical restrictions
which can be imposed on bi-Carleman kernels $\boldsymbol{T}$ without loss of
generality as far as the solving of such equations as \eqref{skequ} is
concerned. These restrictions will be introduced presently by means of
Definitions~\ref{Kmkernel} and \ref{MerKernel}, and the generality will be
preserved using unitary reductions by means of Theorem~\ref{infsmooth} below.
\end{remark}
\par
Throughout this paper, the symbols $\mathbb{C}$, $\mathbb{N}$, and $\mathbb{Z}$, refer to
the complex plane, the set of all positive integers, and the set of all
integers, respectively, and each of the letters $i$ and $j$ is reserved for
all non-negative integers. $C(X,B)$, where $B$ is a Banach space (with norm
$\|\cdot\|_B$), denotes the Banach space (with the norm
$\|f\|_{C(X,B)}=\sup_{x\in X}\,\|f(x)\|_B$) of continuous $B$-valued functions
defined on a locally compact space $X$ and \textit{vanishing at infinity}
(that is, given any $f\in C(X,B)$ and $\varepsilon>0$, there exists a compact
subset $X(\varepsilon,f) \subset X$ such that $\|f(x)\|_{B}<\varepsilon$
whenever $x\not\in X(\varepsilon,f)$). In addition, we introduce the following
notation: if an equivalence class $f\in L^2$ contains a function belonging to
$C(\mathbb{R},\mathbb{C})$, we write $[f]$ to mean that function, and
we denote the $i$-th derivative of $[f]$, if exists, by $[f]^{(i)}$. We also say
that the series $\sum_n f_n$ is \textit{$B$-absolutely convergent in $C(X,B)$}
if $f_n\in C(X,B)$ ($n\in\mathbb{N}$) and the series $\sum_n \|f_n(x)\|_B$
converges in $C(X,\mathbb{R})$ (the sum notation $\sum_n$ will always be used
instead of the more detailed symbol $\sum_{n=1}^\infty$).
\begin{definition}\label{Kmkernel}
A bi-Carleman kernel $\boldsymbol{T}\colon \mathbb{R}^2\to\mathbb{C}$
is called a \textit{$K^\infty$ kernel} (see \cite{nov:CEJM}) if it satisfies the
three generally independent conditions:
\par
(i) the function $\boldsymbol{T}$ and all its partial derivatives of all orders
are in $C\left(\mathbb{R}^2,\mathbb{C}\right)$,
\par
(ii) the Carleman function $\boldsymbol{t}$,
$\boldsymbol{t}(s)=\overline{\boldsymbol{T}(s,\cdot)}$, and its (strong)
derivatives, $\boldsymbol{t}^{(i)}$, of all orders are in
$C\left(\mathbb{R},L^2\right)$,
\par
(iii) the Carleman function $\boldsymbol{t}^{\boldsymbol{\prime}}$,
$\boldsymbol{t}^{\boldsymbol{\prime}}(s)=
\overline{\boldsymbol{T}^{\boldsymbol{\prime}}(s,\cdot)}=
\boldsymbol{T}(\cdot,s)$, and its (strong) derivatives,
$(\boldsymbol{t}^{\boldsymbol{\prime}})^{(i)}$, of all orders are in
$C\left(\mathbb{R},L^2\right)$.
\end{definition}
To deal with $K^\infty$ kernels, we let $D_r^i$ denote the $i$-th order partial
derivative operator with respect to the $r$-th variable, and we let $D_{r_1,r_2}^{i}$
denote the product (in the operator sense) of $i$ factors, each of which is
the mixed second-order partial derivative operator $D_{r_2}^1D_{r_1}^1$, so,
for instance, $D_{r_1,r_2}^{2}=D_{r_2}^1D_{r_1}^1D_{r_2}^1D_{r_1}^1$.
\begin{definition}\label{MerKernel}
Let $\boldsymbol{T}$ be a $K^\infty$ kernel and let $T$ be the integral operator
it induces on $L^2$. We say that the $K^\infty$ kernel $\boldsymbol{T}$ is of
\textit{Mercer type} if every operator belonging to $\mathcal{M}(T)$ is an
integral operator with a $K^\infty$ kernel.
\end{definition}
The idea of a counter-example in \cite[Section~2]{Nov:Lon} may be used
to show that there might be $K^\infty$ kernels which are not of Mercer type. The (in a sense
algebraic) condition on a $K^\infty$ kernel for being of Mercer type is of course tailored
with specific applications in mind. One of these is a theorem, the first main
result of the present paper, which asserts, among other things, that any
$K^\infty$ kernel of Mercer type, along with all its partial and strong
derivatives, is entirely recoverable from the knowledge of at least one
$\mathcal{M}$ factorization for its associated integral operator, by means of
bilinear series formulae universally applicable on arbitrary orthonormal bases
of $L^2$:
\begin{theorem}\label{mfactor}
Let $T\in\mathfrak{R}\left(L^2\right)$ be an integral operator, with a kernel
$\boldsymbol{T}_0$ that is a $K^\infty$ kernel of Mercer type. Then at each
regular value $\lambda\in\Pi(T)$
\par
(a) the Fredholm resolvent $T_\lambda$ of $T$ is also an integral operator
and its kernel, the resolvent kernel $\boldsymbol{T}_\lambda$ for
$\boldsymbol{T}_0$, is also a $K^\infty$ kernel of Mercer type, and, moreover,
\par
(b) for any $\mathcal{M}$ factorization $T = WV^*$ for $T$ and for any
orthonormal basis $\{u_n\}$ for $L^2$, the following formulae hold
\begin{gather}
\left(D_2^jD_1^i\boldsymbol{T}_\lambda\right)(s,t)=
\sum_n\left[R_\lambda(T)Wu_n\right]^{(i)}(s)
\overline{\left[Vu_n\right]^{(j)}(t)}, \label{meijT1}
\\
\begin{split} \label{meijkt2}
\boldsymbol{t}_\lambda^{(i)}(s)&=\sum_n\overline{\left[R_\lambda(T)Wu_n\right]^{(i)}(s)}Vu_n,
\\
\left(\boldsymbol{t}^{\boldsymbol{\prime}}_\lambda\right)^{(j)}(t)&=
\sum_n \overline{\left[Vu_n\right]^{(j)}(t)}R_\lambda(T)Wu_n,
\end{split}
\\
\left[T_\lambda f\right]^{(i)}(s)=
\sum_n\left\langle f,Vu_n\right\rangle_{L^2}\left[R_\lambda(T)Wu_n\right]^{(i)}(s)\label{meijkTf3},
\end{gather}
for all $i$, $j$, all $s$, $t\in\mathbb{R}$, and all $f\in L^2$, where the
series  of \eqref{meijT1} converges $\mathbb{C}$-absolutely in
$C\left(\mathbb{R}^2,\mathbb{C}\right)$, the two series of \eqref{meijkt2},
in which $\boldsymbol{t}_\lambda$, $\boldsymbol{t}^{\boldsymbol{\prime}}_\lambda$
denote the associated Carleman functions of $\boldsymbol{T}_\lambda$, both
converge in $C\left(\mathbb{R},L^2\right)$, and the series of \eqref{meijkTf3}
converges $\mathbb{C}$-absolutely in $C(\mathbb{R},\mathbb{C})$.
\end{theorem}
For the case $\lambda=0$, the results of the theorem have been announced
without proofs in \cite[Theorem~4]{nov:IJPAM1}. In this particular case,
$R_\lambda(T)=I_{L^2}$, $\boldsymbol{T}_\lambda=\boldsymbol{T}_0$, and the
bilinear formula \eqref{meijT1} reminds one of Mercer's (see
\cite[Theorem 4.24]{Porter}) and Kadota's (see \cite{Kadota1}) Theorems; the first
theorem, recall, is about absoluteness and uniformity of convergence of
bilinear (orthogonal) eigenfunction expansions for continuous compactly
supported kernels of positive integral operators, and the second is about
term-by-term differentiability of those expansions while retaining the
absolute and the uniform convergence. The similarity is closest when
Theorem~\ref{mfactor} is applied to a diagonal operator $T$ by taking as
$T=WV^*$ the $\mathcal{M}$ factorization of Example~\ref{TWV} and as $\{u_n\}$
any orthonormal basis with respect to which $T$ is diagonalizable, because
then formula \eqref{meijT1} (with $\lambda=0$) reduces, after a computation,
to a bilinear eigenfunction expansion,
\begin{equation*}
\left(D_2^jD_1^i\boldsymbol{T}_0\right)(s,t)=
\sum_n\lambda_n\left[u_n\right]^{(i)}(s)
\overline{\left[u_n\right]^{(j)}(t)},
\end{equation*}
converging $\mathbb{C}$-absolutely in $C\left(\mathbb{R}^2,\mathbb{C}\right)$.
Applied within the same setting, formula \eqref{meijkTf3} looks like:
$
\left[Tf\right]^{(i)}(s)=
\sum_n\lambda_n \left\langle f,u_n\right\rangle_{L^2}\left[u_n\right]^{(i)}(s)
$
in the sense of $\mathbb{C}$-absolute convergence in
$C\left(\mathbb{R},\mathbb{C}\right)$, and gives something very akin to
Schmidt's Theorem, see \cite[Theorem 4.22]{Porter}.
In the general non-diagonalizable case, it is therefore natural to view
formulae \eqref{meijT1}, \eqref{meijkTf3}
as a ``necessary substitute'' for the above classical diagonal
ones, justifying the name chosen for the kernels defined in
Definitions~\ref{Kmkernel} and \ref{MerKernel}.
\par
Formulae \eqref{meijT1}, \eqref{meijkt2} do not, of course, solve the problem
(mentioned at the end of Remark~\ref{rem1}) of explicitly constructing the
resolvent kernel $\boldsymbol{T}_\lambda$ for $\boldsymbol{T}_0$, but succeed
in reducing it to one of explicitly finding at most countably many functions
$R_\lambda(T)Wu_n$, $Vu_n$ ($n\in\mathbb{N}$), which we do not yet know how to 
resolve. Nevertheless, the formulae in Theorem~\ref{mfactor} may (hopefully)
prove quite interesting from a theoretical perspective in terms of developing
a general scheme for bilinearly representing integral kernels
via the use of the operators that they define but ignoring, if needed, the explicit
knowledge about the spectra of those operators.
\par
The proof of Theorem~\ref{mfactor} is given in Section~\ref{prmfactor} below
and actually proves a somewhat looser version of it, which will be formulated
later, in Section~\ref{faccol}, as Corollary~\ref{mfactor3}. This corollary is
of the same sort as Theorem~\ref{mfactor}, but, among others, we have relaxed
the $\mathcal{M}$ factoring assumption \eqref{f2} about the operators $W$, $V$
being used in formulae \eqref{meijT1}-\eqref{meijkTf3}. In Section~\ref{faccol}, we
also try to give a unified view of various bilinear expansion theorems which
involve the use of canonical forms of operators, by deriving them as
consequences of (the proof of) Theorem~\ref{mfactor}.
\par
At first glance it may seem that the conditions defining $K^\infty$ kernel of
Mercer type are not only hardly verifiable, but also very artifical and rather
contrived, and also that for most applications it is too restrictive to confine
oneself to such kernels. It is somewhat surprising, therefore, to discover
the fact that any bi-integral operator can be made to have as its kernel
a $K^\infty$ kernel of Mercer type. For the interpretation of the corresponding
result, it is helpful to recall the notion of a unitary equivalence. A bounded
linear operator $U\colon {\mathcal{H}}\to L^2$ is said to be \textit{unitary} if
$\mathrm{Ran\,}U=L^2$ and
$\langle Uf,Ug\rangle_{L^2}=\langle f,g\rangle_{\mathcal{H}}$ for all $f$,
$g\in {\mathcal{H}}$. An operator $S\in \mathfrak{R}({\mathcal{H}})$ is said
to be \textit{unitarily equivalent} to an operator
$T\in \mathfrak{R}\left(L^2\right)$ if a unitary operator
$U\colon {\mathcal{H}}\to L^2$ exists such that $T=USU^{-1}$. It is also relevant to
mention the fact that a necessary and sufficient condition that an operator
$S\in\mathfrak{R}({\mathcal{H}})$ be unitarily equivalent to a (general)
bi-Carleman integral operator is that there exist an orthonormal sequence
$\left\{e_n\right\}$ in $\mathcal{H}$ such that
\begin{equation}\label{kh}
\left\|Se_n\right\|_{\mathcal{H}}\rightarrow 0,\quad
\left\|S^*e_n\right\|_{\mathcal{H}}\rightarrow 0\quad
\text{as $n\rightarrow\infty$}
\end{equation}
(or, equivalently, that $0$ belong to the essential spectrum of $SS^*+S^*S$).
This fact was first stated by von Neumann in \cite{Neu} for self-adjoint
operators and was then extended by Korotkov to the general case (see
\cite[Theorem~III.2.7]{Kor:book1}, \cite[Theorem~15.14]{Halmos:Sun}). Recall
\cite[Theorem~15.11]{Halmos:Sun} that the class of operators satisfying
\eqref{kh} includes all bi-integral operators when the Hilbert space
$\mathcal{H}$ is $L^2$, or in general $L^2(Y,\mu)$ associated with a positive,
$\sigma$-finite, separable, and not purely atomic, measure $\mu$. The
bi-integral operators, on the other hand, are generally involved in
second-kind integral equations (like \eqref{skequ}) in $L^2(Y,\mu)$, as the 
adjoint equations to such equations are customarily required to be integral.
\par
It is pleasant to know that the same condition as \eqref{kh} proves necessary
and sufficient for the operator $S$ to be unitarily equivalent to a bi-Carleman
operator generated by a $K^\infty$ kernel of Mercer type.
The second principal result of the present paper, Theorem~\ref{infsmooth}
below, both states this fact and characterizes families incorporating those operators
in $\mathfrak{R}({\mathcal{H}})$ that can be simultaneously transformed by the
same unitary equivalence transformation into bi-Carleman integral operators
having as kernels $K^\infty$ kernels of Mercer type.
\begin{theorem}\label{infsmooth}
Suppose that for an operator family $\mathcal{S}=
\left\{S_\gamma\right\}_{\gamma\in\mathcal{G}}\subset\mathfrak{R}({\mathcal{H}})$
with an index set of arbitrary cardinality there exists an orthonormal
sequence $\left\{e_n\right\}$ in $\mathcal{H}$ such that
\begin{equation}\label{1.2}
\lim\limits_{n\to\infty}\sup\limits_{\gamma\in\mathcal{G}}\left\|S_\gamma e_n\right\|_{\mathcal{H}}=0,
\quad
\lim_{n\to\infty}\sup_{\gamma\in\mathcal{G}}\left\|\left(S_\gamma\right)^* e_n\right\|_{\mathcal{H}}=0.
\end{equation}
Then there exists a unitary operator $U\colon \mathcal{H}\to L^2$ such that all the
operators $T_\gamma=U S_\gamma U^{-1}$ $(\gamma\in\mathcal{G})$ and their
linear combinations are bi-Carleman operators on $L^2$, whose kernels are
$K^\infty$ kernels of Mercer type.
\end{theorem}
This result has recently been published without proof in \cite[Theorem~3]{nov:IJPAM1}.
Section~\ref{intrepr} of the present paper is entirely devoted to proving
Theorem~\ref{infsmooth}. The method of proof yields a technique for
constructing that unitary operator $U\colon \mathcal{H}\to L^2$ whose existence the
theorem asserts. The technique uses no spectral properties of the operators
$S_\gamma$, other than their joint property imposed in \eqref{1.2}, to
determine the action of $U$ by specifying two orthonormal bases, of
$\mathcal{H}$ and of $L^2$, one of which is meant to be the image by $U$ of
the other, the basis for $L^2$ may be chosen to be an infinitely smooth
wavelet basis.
\section{Proof of Theorem~\ref{mfactor}}\label{prmfactor}
(a) Use \eqref{MT}, \eqref{eqresf}, and the invertibility of
$R_\lambda(T)$, to see that

\newpage

\begin{multline*}
\mathcal{M}(T_\lambda)=\left(TR_\lambda(T)\mathfrak{R}\left(L^2\right)
\cup T^*\left(R_\lambda(T)\right)^*\mathfrak{R}\left(L^2\right)\right)
\\
\cap \left(\mathfrak{R}\left(L^2\right)R_\lambda(T)T\cup
\mathfrak{R}\left(L^2\right)\left(R_\lambda(T)\right)^*T^*\right)=
\mathcal{M}(T),
\end{multline*}
for each $\lambda\in\Pi(T)$. Now the first assertion in the theorem is
immediate from Definition~\ref{MerKernel}.
\par
(b) This part of the proof proves the second assertion in the theorem,
and is divided into four steps. The first three steps are to establish formulae
\eqref{meijT1}-\eqref{meijkTf3} for the case in which $\lambda=0$. Step~4 takes
care of the case of an arbitrary regular $\lambda$.
\par
Throughout what follows let $\left\{u_n\right\}$ be an arbitrary but fixed
orthonormal basis for $L^2$, and let $W$, $V$ be arbitrary but likewise fixed
operators of $\mathfrak{R}\left(L^2\right)$ such that both $F=VV^*$ and $G=WW^*$
are in $\mathcal{M}^{+}(T)$, and $T=WV^*$. We then let $\boldsymbol{F}$,
$\boldsymbol{G}$ denote the $K^\infty$ kernels of the integral positive
operators $F$, $G$, respectively.
\par
\textsl{Step 1.} A convenient way to begin the proof of assertion (b)
for $\lambda=0$ is to assume for the moment that the following properties of
the function systems $\{Vu_n\}$, $\{Wu_n\}$ hold true:
\par
(A) $\left[Vu_n\right]^{(i)}$,
$\left[Wu_n\right]^{(i)}\in C(\mathbb{R},\mathbb{C})$, for all
$n\in\mathbb{N}$ and all $i$,
\par
(B) the series $\sum_n\left|\left[Vu_n\right]^{(i)}\right|^2$,
$\sum_n\left|\left[Wu_n\right]^{(i)}\right|^2$ converge in $C(\mathbb{R},\mathbb{C})$, for all $i$.
\par
Having made these assumptions, the first thing to do is to establish the
existence of a $K^\infty$ kernel $\boldsymbol{H}$, with associated Carleman
functions $\boldsymbol{h}$ and $\boldsymbol{h}^{\boldsymbol{\prime}}$, such
that, for all $i$, $j$,
\begin{equation}\label{firsts}
\left(D_2^jD_1^i\boldsymbol{H}\right)(s,t)=
\sum_n\left[Wu_n\right]^{(i)}(s)\overline{\left[Vu_n\right]^{(j)}(t)}
\end{equation}
with the series converging $\mathbb{C}$-absolutely in
$C\left(\mathbb{R}^2,\mathbb{C}\right)$, and
\begin{equation}\label{seconds}
\boldsymbol{h}^{(i)}(s)=\sum_n\overline{\left[Wu_n\right]^{(i)}(s)}Vu_n,
\quad\left(\boldsymbol{h}^{\boldsymbol{\prime}}\right)^{(j)}(t)=\sum_n \overline{\left[Vu_n\right]^{(j)}(t)}Wu_n,
\end{equation}
in the sense of convergence in $C\left(\mathbb{R},L^2\right)$. For this
purpose, invoke the inequalities
\allowdisplaybreaks
\begin{equation*}
\begin{gathered}
\left({\sum_{n=p}^r\left|\left[Wu_n\right]^{(i)}(s)
\overline{\left[Vu_n\right]^{(j)}(t)}\right|}\right)^2
\leq \sum_{n=p}^r\left|\left[Wu_n\right]^{(i)}(s)\right|^2\sum_{n=p}^{r}
\left|\left[Vu_n\right]^{(j)}(t)\right|^2,
\\
\left\|\sum_{n=p}^r\overline{\left[Wu_n\right]^{(i)}(s)}Vu_n\right\|_{L^2}^2\le\|V\|^2
\sum_{n=p}^r\left|\left[Wu_n\right]^{(i)}(s)\right|^2,
\\
\left\|\sum_{n=p}^r \overline{\left[Vu_n\right]^{(j)}(t)}Wu_n\right\|_{L^2}^2\le\|W\|^2
\sum_{n=p}^r\left|\left[Vu_n\right]^{(j)}(t)\right|^2
\end{gathered}
\end{equation*}
to infer, via (A) and (B), that the series of \eqref{firsts} and of \eqref{seconds}
do indeed converge in the senses above. Then apply the corresponding theorems
on termwise differentiation of series to conclude that functions
$\boldsymbol{H}$ of $C\left(\mathbb{R}^2,\mathbb{C}\right)$ and
$\boldsymbol{h}$, $\boldsymbol{h}^{\boldsymbol{\prime}}$ of
$C\left(\mathbb{R},L^2\right)$, defined as
\allowdisplaybreaks
\begin{equation}\label{me3.21}
\begin{gathered}
\boldsymbol{H}(s,t)=\sum_n\left[Wu_n\right](s)\overline{\left[Vu_n\right](t)},
\\
\boldsymbol{h}(s)=\overline{\boldsymbol{H}(s,\cdot)}=\sum_n\overline{
\left[Wu_n\right](s)}Vu_n,
\\
\boldsymbol{h}^{\boldsymbol{\prime}}(t)=\boldsymbol{H}(\cdot,t)=
\sum_n \overline{\left[Vu_n\right](t)}Wu_n,
\end{gathered}
\end{equation}
have the desired expansions \eqref{firsts}, \eqref{seconds} for all $i$, $j$,
and hence that
\begin{equation}\label{meconi}
D_2^jD_1^i\boldsymbol{H}\in C\left(\mathbb{R}^2,\mathbb{C}\right),\quad
\boldsymbol{h}^{(i)},\ \left(\boldsymbol{h}^{\boldsymbol{\prime}}\right)^{(j)}\in
C\left(\mathbb{R},L^2\right)
\end{equation}
for all $i$, $j$. Since, moreover, each mixed partial derivative of
$\left[Wu_n\right](s)\overline{\left[Vu_n\right](t)}$ is everywhere
independent of the sequence in which the partial differentiations with respect
to $s$ and $t$ are carried out, it follows that not only those of the form as
in \eqref{meconi} but also all other partial derivatives of $\boldsymbol{H}$
belong to $C\left(\mathbb{R}^2,\mathbb{C}\right)$, hereby showing
conclusively that $\boldsymbol{H}$ is a $K^\infty$ kernel.
\par
Now fix any $f\in L^2$, and observe then that
\begin{equation}\label{Tf}
Tf=WV^*f=\sum_n\left\langle f,Vu_n\right\rangle_{L^2}Wu_n,
\end{equation}
where the series converges to $Tf$ in $L^2$. On the other hand, the
convergence properties of the series of \eqref{me3.21} make it possible
to write, for each temporarily fixed $s\in\mathbb{R}$, the following chain of
relations
\begin{multline}\label{Kf}
\sum_n\left\langle f,Vu_n\right\rangle_{L^2}\left[Wu_n\right](s)
=\left\langle f,\sum_n\overline{\left[Wu_n\right](s)}Vu_n\right\rangle_{L^2}\\
=
\int_{\mathbb{R}}\left(\sum_n\left[Wu_n\right](s)\overline{\left[Vu_n\right](t)}\right)f(t)\,dt=
\int_{\mathbb{R}}\boldsymbol{H}(s,t)f(t)\,dt.
\end{multline}
Because of the assumptions (A), (B) made about the functions
$\left[Wu_n\right]$, and also because of the inequality
\begin{equation*}
\left(\sum_{n=p}^r\left|\left\langle f,Vu_n\right\rangle_{L^2}\left[Wu_n\right]^{(i)}(s)\right|\right)^2
\leq\sum_{n=p}^r\left|\left\langle V^*f,u_n\right\rangle_{L^2}\right|^2
\sum_{n=p}^r\left|\left[Wu_n\right]^{(i)}(s)\right|^2,
\end{equation*}
the first series of \eqref{Kf} is $\mathbb{C}$-absolutely convergent in
$C(\mathbb{R},\mathbb{C})$ and can be differentiated termwise any number of
times while retaining this type of convergence. Comparison of \eqref{Tf} with
\eqref{Kf} shows then that the series representation \eqref{meijkTf3} holds
with $\lambda=0$ for all $i$, and that
\begin{equation*}
(Tf)(s)=\int_{\mathbb{R}}\boldsymbol{H}(s,t)f(t)\,dt\quad
\text{for almost every $s\in\mathbb{R}$}.
\end{equation*}
Since $f\in L^2$ was arbitrary, the latter equality means that the
operator $T$ (which is equal to the Fredholm resolvent $T_{\lambda}$ at
$\lambda=0$) is an integral operator with the function
$\boldsymbol{H}$ as its kernel, so, by the uniqueness of the kernel,
$\boldsymbol{H}=\boldsymbol{T}_0$ in the
$C\left(\mathbb{R}^2,\mathbb{C}\right)$ sense, and
$\boldsymbol{h}=\boldsymbol{t}_0$,
$\boldsymbol{h}^{\boldsymbol{\prime}}=\boldsymbol{t}^{\boldsymbol{\prime}}_0$
in the $C\left(\mathbb{R},L^2\right)$ sense.
The final conclusion is thus that, when $\lambda=0$, assertion
(b) in the theorem will follow once we show that the functions $Vu_n$,
$Wu_n$ ($n\in\mathbb N$) do in fact enjoy properties (A) and
(B). This is the object of the next two steps.
\par
\textsl{Step 2.} This step consists of proving that, under the assumptions
made about $V$ and $W$ prior to Step~1, the series convergence properties
stated in (B) always hold whenever the function properties stated in
(A) are met, that is, that (A) implies (B). The
proof is further given only for the first series of (B), as the proof
for the second series, $\sum_n\left|\left[Wu_n\right]^{(i)}\right|^2$, is
entirely similar.
\par
If $\ell$ is a non-negative integer, $m$ is a positive integer, and
$Q=[a,b]\times[c,d]$ is a compact rectangle of $\mathbb{R}^2$, we define the
three quantities
\begin{gather}
q_1(Q,m;\ell)=\iint\limits_{Q}
\left(\left(D_{1,2}^\ell\boldsymbol{F}\right)(s,t)-
\sum_{n\le m}\left[Vu_n\right]^{(\ell)}(s)\overline{\left[Vu_n\right]^{(\ell)}(t)}\right)\,ds\,dt,
\label{main1}
\\
q_2(Q,m;\ell)=\sum_{n>m} \int\limits_a^b\left[Vu_n\right]^{(\ell)}(s)\,ds
\int\limits_c^d\overline{\left[Vu_n\right]^{(\ell)}(t)}\,dt,
\label{main2}
\\
q_3(Q,m;\ell)=\iint\limits_{Q}\left(\sum_{n>m} \left[Vu_n\right]^{(\ell)}(s)
\overline{\left[Vu_n\right]^{(\ell)}(t)}\right)\,ds\,dt,
\label{main3}
\end{gather}
and prove that, for each $\ell$,
\begin{equation}\label{q1q2}
q_1(Q,m;\ell)=q_2(Q,m;\ell)\quad \text{for all $Q$ and all $m$.}
\end{equation}
First, for this purpose, utilize the $L^2$ representation
\begin{equation*}
Ff=VV^*f=\sum_n\left\langle f,Vu_n\right\rangle_{L^2}Vu_n\quad(f\in L^2)
\end{equation*}
in order to write, for each $Q$ and $m$,
\begin{multline*}
q_1(Q,m;0)
\\
=\sum_n\left\langle Vu_n,\chi_{[a,b]}\right\rangle_{L^2}
\left\langle\chi_{[c,d]},Vu_n\right\rangle_{L^2}-
\sum_{n\le m} \left\langle Vu_n,\chi_{[a,b]}\right\rangle_{L^2}
\left\langle\chi_{[c,d]},Vu_n\right\rangle_{L^2}
\\
=\sum_{n>m}\left\langle Vu_n,\chi_{[a,b]}\right\rangle_{L^2}
\left\langle\chi_{[c,d]},Vu_n\right\rangle_{L^2}
=q_2(Q,m;0),
\end{multline*}
where $\chi_E$ denotes the characteristic function of a set $E$. This implies
that the identity \eqref{q1q2} holds with $\ell=0$. Proceeding by induction
over $\ell$, suppose identity \eqref{q1q2} to be satisfied for some fixed
$\ell$. The stage is now set for the induction step.
\par
It is first necessary to remark that the integrand in \eqref{main1} must be
non-negative on the main diagonal of $\mathbb{R}^2$, that is, denoting the
integrand by $\boldsymbol{F}_m^{\ell}(s,t)$:
\begin{equation}\label{nonneg}
\boldsymbol{F}_m^{\ell}(s,s)=\left(D_{1,2}^\ell\boldsymbol{F}\right)(s,s)-
\sum_{n\le m}\left|\left[Vu_n\right]^{(\ell)}(s)\right|^2
\ge0\quad\text{for all $s\in\mathbb{R}$.}
\end{equation}
Indeed, if it were not so, there would exist a square $Q^\prime$, with centre
at some point on the main diagonal of $\mathbb{R}^2$ and sides parallel to the
coordinate axes, such that its corresponding quantity $q_1(Q^\prime,m;\ell)$
would be negative, contradicting the non-negativity of $q_2(Q^\prime,m;\ell)$
which in turn would be implied by coincidence of the integration intervals on
the right side of \eqref{main2}. In more detail, if
$\boldsymbol{F}_m^{\ell}(s_0,s_0)=-2\delta$ for some $s_0$ and $\delta>0$, then
by the continuity of $\boldsymbol{F}_m^{\ell}$ on $\mathbb{R}^2$, which follows
from that of $D_{1,2}^\ell\boldsymbol{F}$ and from property (A),
there is a square $Q^\prime$,
\begin{equation*}
Q^\prime=\left\{(s,t)\in\mathbb{R}^2: |s-s_0|\le\varepsilon/2,\,|t-s_0|\le
\varepsilon/2\right\}
\quad(\varepsilon>0)
\end{equation*}
in which $\mathrm{Re}\,\boldsymbol{F}_m^{\ell}(s,t)<-\delta$.
By the induction hypothesis
\begin{equation}\label{Q1Q2}
\iint\limits_{Q^\prime}\boldsymbol{F}_m^{\ell}(s,t)\,ds\,dt=
q_1(Q^\prime,m;\ell)=q_2(Q^\prime,m;\ell)\geq0
\end{equation}
so
$0\leq\iint\limits_{Q^\prime}\mathrm{Re}\,\boldsymbol{F}_m^{\ell}(s,t)\,ds\,dt
<-\delta\varepsilon^2<0,$
a contradiction. (Note: that the diagonal value, $\boldsymbol{F}_m^{\ell}(s,s)$, of
the integrand in \eqref{main1} at any point $s\in\mathbb{R}$ and the value of
the integral in \eqref{Q1Q2} taken over any square like
$Q^\prime=[a,b]\times[a,b]$ both have, for any $\ell$ and $m$, a zero
imaginary part comes from the Hermiticity of
$\boldsymbol{F}_m^{\ell}$, $\boldsymbol{F}_m^{\ell}(s,t)=\overline{\boldsymbol{F}_m^{\ell}(t,s)}$
for all $s$, $t\in\mathbb{R}$.
This latter property is mainly inherited from that of $\boldsymbol{F}$
($\boldsymbol{F}(s,t)=\overline{\boldsymbol{F}(t,s)}$ for all $s$,
$t\in\mathbb{R}$ because of the self-adjointness of $F$) thanks to assumption
(i) about $\boldsymbol{F}$, as follows:
$
\left(D_{1,2}^\ell\boldsymbol{F}\right)(s,t)=
\overline{\left(D_{2,1}^\ell\boldsymbol{F}\right)(t,s)}=
\overline{\left(D_{1,2}^\ell\boldsymbol{F}\right)(t,s)}
$
for all $s$, $t\in\mathbb{R}$.)
\par
Further, it is seen from \eqref{nonneg} that, for each $m$, there is the inequality
\begin{equation}\label{mainineq}
\sum_{n\le m}\left|\left[Vu_n\right]^{(\ell)}(s)\right|^2\le
\left(D_{1,2}^\ell\boldsymbol{F}\right)(s,s)\le C_\ell=
\sup_{s\in\mathbb{R}}\left(D_{1,2}^\ell\boldsymbol{F}\right)(s,s)
\quad(s\in\mathbb{R}),
\end{equation}
from which it follows in particular that the series in the integrand of
\eqref{main3} is termwise integrable, implying (via \eqref{main2} and \eqref{q1q2})
that $q_1(Q,m;\ell)=q_3(Q,m;\ell)$ for all $m$ and all $Q$. In turn, this new
identity implies that
\begin{equation}\label{lFst}
\left(D_{1,2}^\ell\boldsymbol{F}\right)(s,t)=\sum_n\left[Vu_n\right]^{(\ell)}(s)
\overline{\left[Vu_n\right]^{(\ell)}(t)}
\end{equation}
almost everywhere in $\mathbb{R}^2$. The series here converges by
\eqref{mainineq} uniformly in each variable separately for all values of the
other, and its sum-function, denoted by $\boldsymbol{S}_\ell$, is therefore a
continuous function of either argument. One can now appeal directly to a
subtle Example~2 from \cite[Ch.~14, pp.~545-546]{Zaanenb}, as applied to
functions $f_1=D_{1,2}^\ell\boldsymbol{F}$ and $f_2=\boldsymbol{S}_\ell$,
in order to be sure that
\begin{equation}\label{Dl12Fss}
\left(D_{1,2}^\ell\boldsymbol{F}\right)(s,s)
=\sum\limits_n\left|\left[Vu_n\right]^{(\ell)}(s)\right|^2\quad \text{for all $s\in\mathbb{R}$.}
\end{equation}
Since $\left(D_{1,2}^\ell\boldsymbol{F}\right)(s,s)\to 0$ as $|s|\to\infty$,
Dini's Monotone Convergence Theorem may now be applied to the 1-point
compactification of $\mathbb{R}$, to yield the conclusion that the series of
\eqref{Dl12Fss} does converge in $C(\mathbb{R},\mathbb{C})$ (compare this
result with property (B)). In particular, it follows that the series
of \eqref{lFst} is converging in $C\left(\mathbb{R}^2,\mathbb{C}\right)$ to
$D_{1,2}^\ell\boldsymbol{F}$. That, in turn, justifies the following
computation
\allowdisplaybreaks
\begin{multline*}
q_1(Q,m;\ell+1)
\\=\left(D_{1,2}^\ell\boldsymbol{F}\right)(b,d)-\left(D_{1,2}^\ell\boldsymbol{F}\right)(b,c)-
\left(D_{1,2}^\ell\boldsymbol{F}\right)(a,d)+\left(D_{1,2}^\ell\boldsymbol{F}\right)(a,c)\\
-\sum_{n\le m}\left(\left[Vu_n\right]^{(\ell)}(b)-\left[Vu_n\right]^{(\ell)}(a)\right)
\left(\overline{\left[Vu_n\right]^{(\ell)}(d)-\left[Vu_n\right]^{(\ell)}(c)}\right)
\\
=\sum_n\left(\left[Vu_n\right]^{(\ell)}(b)\overline{\left[Vu_n\right]^{(\ell)}(d)}
        -\left[Vu_n\right]^{(\ell)}(b)\overline{\left[Vu_n\right]^{(\ell)}(c)}\right.
      \\ \left. -\left[Vu_n\right]^{(\ell)}(a)\overline{\left[Vu_n\right]^{(\ell)}(d)}
        +\left[Vu_n\right]^{(\ell)}(a)\overline{\left[Vu_n\right]^{(\ell)}(c)}\right)
\\-\sum_{n\le m}\left(\left[Vu_n\right]^{(\ell)}(b)-\left[Vu_n\right]^{(\ell)}(a)\right)
\left(\overline{\left[Vu_n\right]^{(\ell)}(d)-\left[Vu_n\right]^{(\ell)}(c)}\right)
\\=\sum_{n>m}\left(\left[Vu_n\right]^{(\ell)}(b)-\left[Vu_n\right]^{(\ell)}(a)\right)
\left(\overline{\left[Vu_n\right]^{(\ell)}(d)-\left[Vu_n\right]^{(\ell)}(c)}\right)
\\=q_2(Q,m;\ell+1),
\end{multline*}
which proves that \eqref{q1q2} with $\ell+1$ instead of $\ell$ holds true for all
$m$ and $Q=[a,b]\times[c,d]$. Therefore, by induction, the identity
\eqref{q1q2} is true for every non-negative integer $\ell$. Hence, as
\eqref{q1q2} implies the $C(\mathbb{R},\mathbb{C})$ convergence of the series
of \eqref{Dl12Fss} by what has just been seen in the course of the induction
step, the first series of (B) converges in $C(\mathbb{R},\mathbb{C})$
for each fixed $i$.
\par
\textsl{Step 3.} In this step, the proof of assertion (b) will be
completed for $\lambda=0$, by showing that, under the conditions laid on
$V$ and $W$ at the beginning of the proof, property (A)
always holds. We shall restrict ourselves to dealing only with the functions
$Vu_n$ ($n\in\mathbb{N}$), because the proof of (A) for $Wu_n$
($n\in\mathbb{N}$) can be obtained in a similar way, but using respectively
$\mathrm{Ran\,}W^*$ and $\boldsymbol{G}$ instead of $\mathrm{Ran\,}V^*$ and
$\boldsymbol{F}$.
\par
Choose $\left\{v_k\right\}$ to be an orthonormal basis for the subspace
$\overline{\mathrm{Ran\,} V^*}$, with the property that
$\left\{v_k\right\}\subset\mathrm{Ran\,}V^*$, and let
$\left\{\widetilde{u}_n\right\}$ be any orthonormal basis for $L^2$ such that
$\left\{v_k\right\}\subset\left\{\widetilde{u}_n\right\}$. Observe that
$v_k\in\mathrm{Ran\,}V^*$ implies $Vv_k=Ff_k$ for some $f_k\in L^2$. Therefore,
by property (ii) for the Carleman function
$\boldsymbol{f}(s)=\overline{\boldsymbol{F}(s,\cdot)}$, there is in the
equivalence class $Vv_k$ a function of $C(\mathbb{R},\mathbb{C})$, namely
$\left[Vv_k\right](s)=\left\langle f_k,\boldsymbol{f}(s)\right\rangle_{L^2}$
($s\in\mathbb{R}$), such that its every derivative,
$\left[Vv_k\right]^{(i)}(\cdot)=
\left\langle f_k,\boldsymbol{f}^{(i)}(\cdot)\right\rangle_{L^2}$,
is also in $C(\mathbb{R},\mathbb{C})$. Then, since every $V\widetilde{u}_n$ is
equal either to $Vv_k$, for some $k$, or to the zero function and
therefore all $V\widetilde{u}_n$ satisfy (A),
the reasoning of Step~2 can be applied to the function system
$\{V\widetilde{u}_n\}$ in place of $\{Vu_n\}$ in order to arrive at the
conclusion that, for each $i$, the series
$\sum_n\left|\left[V\widetilde{u}_n\right]^{(i)}\right|^2$ converges in
$C(\mathbb{R},\mathbb{C})$ (cf. (B)). After that, one can conclude
immediately that each series of the form
$\sum_n\left\langle f,\widetilde{u}_n\right\rangle_{L^2}\left[V\widetilde{u}_n\right]^{(i)}$
$\left(f\in L^2\right)$ also converges in $C(\mathbb{R},\mathbb{C})$, and
hence defines a $C(\mathbb{R},\mathbb{C})$ function, which is nothing else
than $\left[Vf\right]^{(i)}$ because
$Vf=\sum_n\left\langle f,\widetilde{u}_n\right\rangle_{L^2}V\widetilde{u}_n$
in the sense of convergence in $L^2$. In particular,
$\left[Vu_n\right]^{(i)}\in C(\mathbb{R},\mathbb{C})$, for all
$n\in\mathbb{N}$ and all $i$.
\par
\textsl{Step 4.}
Now let $\lambda$ be an arbitrary non-zero regular value for $T$, and
factorize the Fredholm resolvent $T_\lambda$ of $T$ at $\lambda$ in the form
$T_\lambda=W_\lambda V^*$ where $W_\lambda=R_\lambda(T)W$ (see \eqref{eqresf},
\eqref{f1}). This factorization need not be an $\mathcal{M}$ factorization for
$T_\lambda$. If, however, the operator
$G_\lambda=W_\lambda\left(W_\lambda\right)^*$ is known to be integral with a
$K^\infty$ kernel, then the previous three steps of the proof may easily be
adapted, with $W_\lambda$ written instead of $W$, to show that the formulae
\eqref{meijT1}-\eqref{meijkTf3} all hold exactly as stated in the theorem.
\par
Let us therefore focus attention on the operator $G_\lambda$. That this
operator is Carleman follows form the representation
\begin{multline}\label{frres}
G_\lambda=W_\lambda \left(W_\lambda\right)^*
=G +\lambda TR_\lambda(T)G+\bar\lambda G\left(R_\lambda(T)\right)^*T^*
\\
+\left|\lambda\right|^2TR_\lambda(T) G\left(R_\lambda(T)\right)^*T^*
\end{multline}
which may be established with the aid of the equality
$R_\lambda(T)=I_{L^2}+\lambda TR_\lambda(T)$ (cf. \eqref{eqress}), and in which each term is a
Carleman operator by the right-multiplication lemma. In addition, if
$\boldsymbol{g}$ is the Carleman function corresponding with the $K^\infty$
kernel $\boldsymbol{G}$ of the integral operator $G=WW^*$, then the function
$\boldsymbol{g}_\lambda\colon \mathbb{R}\to L^2$ defined by
\begin{multline*}
\boldsymbol{g}_\lambda(\cdot)=\boldsymbol{g}(\cdot)+\bar\lambda
\left(R_\lambda(T) G\right)^*\left(\boldsymbol{t}_0(\cdot)\right)+
\lambda\left(\left(R_\lambda(T)\right)^*T^*\right)^*
\left(\boldsymbol{g}(\cdot)\right)
\\+
\left|\lambda\right|^2\left(R_\lambda(T) G\left(R_\lambda(T)\right)^*T^*\right)^*
\left(\boldsymbol{t}_0(\cdot)\right)
\end{multline*}
can be regarded as a Carleman function associated with the Carleman kernel
of $G_\lambda$, by Proposition~\ref{rimlt} again. Since, for each $i$,
\begin{equation}\label{mecfder}
\boldsymbol{g}^{(i)}, \boldsymbol{t}_0^{(i)},
\left(\boldsymbol{t}^{\boldsymbol{\prime}}_0\right)^{(i)}\in C\left(\mathbb{R},L^2\right),
\end{equation}
it follows that
\begin{multline}\label{mederqp}
\boldsymbol{g}_\lambda^{(i)}(\cdot)=\boldsymbol{g}^{(i)}(\cdot)+\bar\lambda
\left(R_\lambda(T) G\right)^*\left(\boldsymbol{t}_0^{(i)}(\cdot)\right)+
\lambda\left(\left(R_\lambda(T)\right)^*T^*\right)^*
\left(\boldsymbol{g}^{(i)}(\cdot)\right)
\\+
\left|\lambda\right|^2\left(R_\lambda(T)G\left(R_\lambda(T)\right)^*T^*\right)^*
\left(\boldsymbol{t}_0^{(i)}(\cdot)\right)\in C\left(\mathbb{R},L^2\right),
\end{multline}
because the operators $R_\lambda(T)$, $G$, and $T$, are bounded. Analogously,
for each $i$,
\begin{equation}
\begin{gathered}\label{meforpart}
\left(\left(R_\lambda(T)\right)^*(\boldsymbol{t}_0(\cdot))\right)^{(i)}=
\left(R_\lambda(T)\right)^*\left(\boldsymbol{t}_0^{(i)}(\cdot)\right)
\in C\left(\mathbb{R},L^2\right),\\
\left(R_\lambda(T)(\boldsymbol{g}(\cdot))\right)^{(i)}=
R_\lambda(T)\left(\boldsymbol{g}^{(i)}(\cdot)\right)\in
C\left(\mathbb{R},L^2\right),\\
R_\lambda(T) G\left(R_\lambda(T)\right)^*
\left(\boldsymbol{t}_0(\cdot))\right)^{(i)}=
R_\lambda(T) G\left(R_\lambda(T)\right)^*
\left(\boldsymbol{t}_0^{(i)}(\cdot)\right)
\in C\left(\mathbb{R},L^2\right).
\end{gathered}
\end{equation}
Note also that each of the last three terms in \eqref{frres} is a product of
two Carleman operators. From this it becomes possible to express the inducing
Carleman kernel $\boldsymbol{G}_\lambda$ of $G_\lambda$ by means of
convolutions as follows:
\begin{multline*}
\boldsymbol{G}_\lambda(s,t)=
\boldsymbol{G}(s,t)+\lambda\left\langle\boldsymbol{g}(t),
\left(R_\lambda(T)\right)^*\left(\boldsymbol{t}_0(s)\right)\right\rangle_{L^2}+\bar\lambda
\left\langle\boldsymbol{t}^{\boldsymbol{\prime}}_0(t),R_\lambda(T)(\boldsymbol{g}(s))\right\rangle_{L^2}
\\+
\left|\lambda\right|^2\left\langle\boldsymbol{t}^{\boldsymbol{\prime}}_0(t),R_\lambda(T) G
\left(R_\lambda(T)\right)^*\left(\boldsymbol{t}_0(s)\right)\right\rangle_{L^2}.
\end{multline*}
Then, for each fixed $i$, $j$, a partial differentiation of
$\boldsymbol{G}_\lambda$ yields, after the equalities of \eqref{meforpart} are
taken into account, the following equality on $\mathbb{R}^2$:
\begin{multline*}
\left(D_2^jD_1^i\boldsymbol{G}_\lambda\right)(s,t)=
\left(D_2^jD_1^i\boldsymbol{G}\right)(s,t)
+
\lambda\left\langle\boldsymbol{g}^{(j)}(t),
\left(R_\lambda(T)\right)^*\left(\boldsymbol{t}_0^{(i)}(s)\right)\right\rangle_{L^2}
\\+
\bar\lambda\left\langle\left(\boldsymbol{t}^{\boldsymbol{\prime}}_0\right)^{(j)}(t),R_\lambda(T)
\left(\boldsymbol{g}^{(i)}(s)\right)\right\rangle_{L^2}
\\+
\left|\lambda\right|^2\left\langle\left(\boldsymbol{t}^{\boldsymbol{\prime}}_0\right)^{(j)}(t),
R_\lambda(T) G\left(R_\lambda(T)\right)^*\left(\boldsymbol{t}_0^{(i)}(s)\right)
\right\rangle_{L^2};
\end{multline*}
clearly we are free to compute each term on the right-hand side here by
alternatively applying $D_2^1$ and $D_1^1$ in any order we please. Hence,
according to \eqref{mecfder}, \eqref{meforpart}, and by the continuity
of the inner product, the partial derivatives of $\boldsymbol{G}_\lambda$ are
all in $C\left(\mathbb{R}^2,\mathbb{C}\right)$. This acquisition in
conjunction with \eqref{mederqp} implies that $\boldsymbol{G}_\lambda$ is a
$K^\infty$ kernel, thereby completing the proof of assertion (b) of
the theorem.

\centerline{}
\section{Corollaries and Applications} \label{faccol}
\begin{remark}\label{moreuse}
An easier way to see that $G_\lambda=W_\lambda\left(W_\lambda\right)^*$ in
\eqref{frres} is an integral operator with a $K^\infty$ kernel is by making
more use of the assumption in Theorem~\ref{mfactor} that the $K^\infty$ kernel
$\boldsymbol{T}_0$ of $T$ is of Mercer type, that is, that the operator family
$\mathcal{M}(T)$ consists only of integral operators with $K^\infty$ kernels.
The argument might be as follows: the first and last terms in the right-hand
side of \eqref{frres} clearly both belong to $\mathcal{M}^{+}(T)$
(see \eqref{f2} and Remark~\ref{remmt}). The second term
$\lambda TR_\lambda(T)G$, and hence its adjoint
$\bar\lambda G\left(R_\lambda(T)\right)^*T^*$ (which is just the third one),
does belong to $\mathcal{M}(T)$, because $\lambda TR_\lambda(T)G=
\lambda TR_\lambda(T)B^*T^*\in T\mathfrak{R}\left(L^2\right)
\cap\mathfrak{R}\left(L^2\right)T^*\subset\mathcal{M}(T)$ or
$\lambda TR_\lambda(T)G=\lambda TR_\lambda(T)BT\in
T\mathfrak{R}\left(L^2\right)\cap\mathfrak{R}\left(L^2\right)T\subset
\mathcal{M}(T)$ according as $G=TB$ or $G=BT$ (see again Remark~\ref{remmt}).
Each summand on the right of \eqref{frres} is thus an integral operator with
a $K^\infty$ kernel and hence so is $G_\lambda$ itself.
\end{remark}
In addition to Remark~\ref{moreuse} let us confess that also in the course of
the whole proof of assertion (b) in Theorem~\ref{mfactor} we have not
used the full strength of the assumption on the $K^\infty$ kernel
$\boldsymbol{T}_0$ of being of Mercer type, but only a consequence of it,
namely the existence of inducing $K^\infty$ kernels for $F=VV^*$ and $G=WW^*$,
whenever $T=WV^*$ is an $\mathcal{M}$ factorization for $T$.
So what that proof really proves is the following, more general, result
intended for the case where $T=WV^*$ is not necessarily an $\mathcal{M}$
factorization for, at first, not necessarily integral $T$.
\begin{corollary}\label{mfactor3}
Suppose that an operator $T\in\mathfrak{R}\left(L^2\right)$ has a factorization 
into a product $T=WV^*$ $(V, W\in \mathfrak{R}\left(L^2\right))$ such that both 
$F=VV^*$ and $G=WW^*$ are integral operators with $K^\infty$ kernels. Then at 
each regular value $\lambda\in\Pi(T)$ the Fredholm resolvent $T_\lambda$ of $T$
is an integral operator and its kernel $\boldsymbol{T}_\lambda$ is a $K^\infty$ 
kernel. Moreover, for any orthonormal basis $\{u_n\}$ for $L^2$, the following 
formulae hold
\begin{gather}
\left(D_2^jD_1^i\boldsymbol{T}_\lambda\right)(s,t)=
\sum_n\left[R_\lambda(T)Wu_n\right]^{(i)}(s)
\overline{\left[Vu_n\right]^{(j)}(t)}, \label{meijT13}
\\
\begin{split} \label{meijkt23}
\boldsymbol{t}_\lambda^{(i)}(s)&=\sum_n\overline{\left[R_\lambda(T)Wu_n\right]^{(i)}(s)}Vu_n,
\\
\left(\boldsymbol{t}^{\boldsymbol{\prime}}_\lambda\right)^{(j)}(t)&=
\sum_n \overline{\left[Vu_n\right]^{(j)}(t)}R_\lambda(T)Wu_n,
\end{split}
\\
\left[T_\lambda f\right]^{(i)}(s)=
\sum_n\left\langle f,Vu_n\right\rangle_{L^2}\left[R_\lambda(T)Wu_n\right]^{(i)}(s)\label{meijkTf33},
\end{gather}
for all $i$, $j$, all $s$, $t\in\mathbb{R}$, and all $f\in L^2$, where the
series of \eqref{meijT13} converges $\mathbb{C}$-absolutely in
$C\left(\mathbb{R}^2,\mathbb{C}\right)$, the two series  of \eqref{meijkt23}
both converge in $C\left(\mathbb{R},L^2\right)$, and the series of
\eqref{meijkTf33} converges $\mathbb{C}$-absolutely in $C(\mathbb{R},\mathbb{C})$.
\end{corollary}
The proof is exactly the same as the proof given in the previous section for 
assertion (b) of Theorem~\ref{mfactor}. Corollary~\ref{mfactor3}
also opens a slightly different, but equivalent, way to define a $K^\infty$ kernel
of Mercer type, involving families $\mathcal{M}^{+}(\cdot)$ in place of
$\mathcal{M}(\cdot)$ (cf. Definition~\ref{MerKernel}):
\begin{corollary}\label{cormer}
An operator $T\in\mathfrak{R}\left(L^2\right)$ is an integral operator with a
$K^\infty$ kernel of Mercer type if and only if the family $\mathcal{M}^{+}(T)$
consists only of integral operators with $K^\infty$ kernels.
\end{corollary}
\begin{proof}
The ``only if'' is immediate  from the inclusion
$\mathcal{M}^{+}(T)\subset\mathcal{M}(T)$. For the proof of the ``if'', assume
that every operator $A\in\mathcal{M}^{+}(T)$ is an integral operator with a
$K^\infty$ kernel. If $S\in\mathcal{M}(T)$, it is to be proved that $S$ is an
integral operator with a $K^\infty$ kernel. For this, let $S=WV^*$ be an
$\mathcal{M}$ factorization for $S$, where, according to \eqref{f2} and
Remark~\ref{remmt}, $VV^*$, $WW^*\in\mathcal{M}^{+}(S)\subseteq\mathcal{M}^{+}(T)$.
Then it follows by assumption that both $F=VV^*$ and $G=WW^*$ are integral
operators with $K^\infty$ kernels. Apply Corollary~\ref{mfactor3} to the
operator $S=WV^*$ and when $\lambda=0$ conclude that $S$ is an integral
operator with a $K^\infty$ kernel. The corollary is proved.
\end{proof}
\begin{remark}\label{rDost}
The next result, Dostani\'c's \cite{Dost89}-\cite{Dost93} extension of Mercer's
Theorem to a class of continuous non-Hermitian kernels on $[0,1]^2$,
deserves mention because it turns out to also fit into the general bilinear
expansion scheme proved in Section~\ref{prmfactor}.
\begin{proposition}\label{pureDost} If $T=H\left(I_{L^2(0,1)}+S\right)$, where
$H\ge0$, $S=S^*$ are integral operators induced on $L^2(0,1)$ by continuous
kernels on $[0,1]^2$, and if either {\rm(I)} $I_{L^2(0,1)}+S$ is positive and
invertible, or {\rm(II)} $I_{L^2(0,1)}+S$ is invertible and $H$ is one-to-one,
then the kernel $\boldsymbol{T}_0$ of the integral operator $T$ is represented
on $[0,1]^2$ by the absolutely and uniformly convergent series
\begin{equation}\label{Dost:form}
\boldsymbol{T}_0(s,t)=\sum_n \lambda_n\psi_n(s)\overline{\varphi_n(t)},
\end{equation}
where $\{\psi_n\}$ and $\{\varphi_n\}$ are biorthogonal systems of
eigenfunctions for $T$ and for $T^*$, respectively:
$T\psi_n=\lambda_n\psi_n$, $T^*\varphi_m=\lambda_m\varphi_m$,
$\langle\psi_n,\varphi_m\rangle_{L^2(0,1)}=\delta_{nm}$
(Kronecker delta).
\end{proposition}
The proof of this result rests on an eigenvalue-eigenfunction analysis
(see, e.g., \cite[Chapter~5, \S8]{Gohberg:Krejn}) of compact operators
that are self-adjoint with respect to the definite or indefinite inner
product $[f,g]_{L^2(0,1)}=\langle(I_{L^2(0,1)}+S)f,g\rangle_{L_2(0,1)}$
according as case (I) or case (II) is in question.
Alternatively, the result may be proved without direct recourse to that
analysis by using a factorization argument similar to that used for
Theorem~\ref{mfactor}(b);
a possible outline of a proof may be roughly sketched as follows.
First note that the integral operator $T$ in Proposition~\ref{pureDost} has an
$\mathcal{M}$ factorization $T=WV^*$ such that both $VV^*$ and $WW^*$ are
integral operators with continuous kernels on $[0,1]^2$. Indeed, define the
factors by $V=(I_{L^2(0,1)}+S)^{\frac12}\Lambda^{\frac12}$,
$W=(I_{L^2(0,1)}+S)^{-\frac12}\Lambda^{\frac12}$,
where $\Lambda=(I_{L^2(0,1)}+S)^{\frac12}H(I_{L^2(0,1)}+S)^{\frac12}$, in case
(I), or by $V=(I_{L^2(0,1)}+S)H^{\frac12}$, $W=H^{\frac12}$ in case (II);
in both cases, then, $VV^*=(I_{L^2(0,1)}+S)T$, $WW^*=T(I_{L^2(0,1)}+S)^{-1}=H$,
and $T=WV^*$, where the integral operators $H$, $S$, and $T$, are known to
have continuous kernels. Having thus factorized the integral operator $T$,
it can be proved by adapting arguments in Section~\ref{prmfactor} that any
series of the form $\sum_nWu_n(s)\overline{Vu_n(t)}$, where $\{u_n\}$ is an
orthonormal basis in $L^2(0,1)$, converges absolutely and uniformly to the
kernel $\boldsymbol{T}_0$ of $T$. Thus, in order to prove the desired convergence 
behavior of representation
\eqref{Dost:form} in Proposition~\ref{pureDost}, it suffices to prove that
$\lambda_n\psi_n(s)\overline{\varphi_n(t)}=Wu_n(s)\overline{Vu_n(t)}$
($n\in\mathbb{N}$) for some orthonormal basis $\{u_n\}$. 
It is a straightforward calculation to show that in case (I) (resp., (II)) such a $\{u_n\}$ can be
chosen to be an orthonormal basis in $L^2(0,1)$ with respect to which the
 compact, self-adjoint operator
$(I_{L^2(0,1)}+S)^{\frac12}H(I_{L^2(0,1)}+S)^{\frac12}$
(resp., $H^{\frac12}(I_{L^2(0,1)}+S)H^{\frac12}$) diagonalizes (see proof of
Corollary~\ref{gDost} below for more details).
\end{remark}
The following generalization of Proposition~\ref{pureDost} is included as
an application of Corollary~\ref{mfactor3}.
\begin{corollary}\label{gDost}
Let $T=H\left(I_{L^2}+S\right)$, where $0\le H\in\mathfrak{R}\left(L^2\right)$,
$S\in\mathfrak{R}(L^2)$, and both $H$ and $(I_{L^2}+S^*)H(I_{L^2}+S)$ are integral
operators with $K^\infty$ kernels. Then at each regular value
$\lambda\in\Pi(T)$ the Fredholm resolvent $T_\lambda$ of $T$ is an integral
operator whose kernel $\boldsymbol{T}_\lambda$ is a $K^\infty$ kernel. If, in
addition, $\Lambda=H^{\frac12}(I_{L^2}+S)H^{\frac12}$ is a diagonal operator with
diagonal entries $\lambda_1,\lambda_2,\lambda_3,\dots$, then there are
functions  $\psi_n$, $\varphi_n$ ($n\in\mathbb{N}$) in $L^2$ for which the
following formulae hold: $T\psi_n=\lambda_n\psi_n$,
$T^*\varphi_m=\bar\lambda_m\varphi_m$, $\langle\psi_n,\varphi_m\rangle_{L^2}
=\lambda_n\delta_{nm}$ for all $n$, $m\in\mathbb{N}$, and
\begin{gather}
\left(D_2^jD_1^i\boldsymbol{T}_\lambda\right)(s,t)=
\sum_n\frac1{1-\lambda\lambda_n}\left[\psi_n\right]^{(i)}(s)
\overline{\left[\varphi_n\right]^{(j)}(t)}, \label{meijT14}
\\
\begin{split} \label{meijkt24}
\boldsymbol{t}_\lambda^{(i)}(s)&=\sum_n
\frac1{1-\overline{\lambda}\overline{\lambda_n}}
\overline{\left[\psi_n\right]^{(i)}(s)}\varphi_n,
\\
\left(\boldsymbol{t}^{\boldsymbol{\prime}}_\lambda\right)^{(j)}(t)&=
\sum_n\frac1{1-\lambda\lambda_n}
\overline{\left[\varphi_n\right]^{(j)}(t)}\psi_n,
\end{split}
\\
\left[T_\lambda f\right]^{(i)}(s)=
\sum_n\frac1{1-\lambda\lambda_n}\left\langle f,\varphi_n\right\rangle_{L^2}
\left[\psi_n\right]^{(i)}(s)\label{meijkTf34},
\end{gather}
for all $\lambda\in\Pi(T)$, all $i$, $j$, and all $s$, $t\in\mathbb{R}$, and
all $f\in L^2$,  where the series of \eqref{meijT14} converges 
$\mathbb{C}$-absolutely in $C\left(\mathbb{R}^2,\mathbb{C}\right)$, the two 
series  of \eqref{meijkt24} both converge in $C\left(\mathbb{R},L^2\right)$, 
and the series of \eqref{meijkTf34} converges $\mathbb{C}$-absolutely in
$C(\mathbb{R},\mathbb{C})$.
\end{corollary}
\begin{proof} Put $W=H^{\frac12}$, $V=(I_{L^2}+S^*)H^{\frac12}$, and
let the orthonormal basis $\{u_n\}$ for $L^2$ allow the diagonal operator
$\Lambda=H^{\frac12}(I_{L^2}+S)H^{\frac12}$ to be written as
\begin{equation}\label{LAMB}
\Lambda=\sum_n \lambda_n\left\langle\cdot,u_n\right\rangle_{L^2} u_n.
\end{equation}
Since, by assumption, both operators $WW^*=H$ and $VV^*=(I_{L^2}+S^*)H(I_{L^2}+S)$
are integral and are defined by $K^\infty$ kernels, Corollary~\ref{mfactor3}
can be applied with respect to factorization $T=H(I_{L^2}+S)=WV^*$ and basis
$\{u_n\}$ to conclude that the Fredholm resolvent $T_\lambda$ of $T$
is an integral operator, with a kernel $\boldsymbol{T}_\lambda$ that is a
$K^\infty$ kernel for which the expansion formulae
\eqref{meijT13}-\eqref{meijkTf33} hold. By the change of notation:
$\psi_n=Wu_n$, $\varphi_n=Vu_n$ ($n\in\mathbb{N}$),
these formulae can be rewritten in forms \eqref{meijT14}-\eqref{meijkTf34},
respectively, because then
$T\psi_n=H(I_{L^2}+S)H^{\frac12}u_n=H^{\frac12}\Lambda u_n=\lambda_n \psi_n$
by \eqref{LAMB}, and therefore
$R_\lambda(T)\psi_n=\frac1{1-\lambda\lambda_n}\psi_n$. Moreover,
$T^*\varphi_n=(I_{L^2}+S^*)HVu_n=(I_{L^2}+S^*)H(I_{L^2}+S^*)H^{\frac12}u_n
=(I_{L^2}+S^*)H^{\frac12}\Lambda^*u_n=\bar\lambda_n \varphi_n$
and  $\left\langle\psi_n,\varphi_m\right\rangle_{L^2}=
\left\langle H^{\frac12}u_n,(I_{L^2}+S^*)H^{\frac12}u_m\right\rangle_{L^2}
=\left\langle\Lambda u_n,u_m\right\rangle_{L^2}
=\lambda_n\left\langle u_n,u_m\right\rangle_{L^2}=\lambda_n\delta_{nm}$
for all $m$, $n\in\mathbb{N}$. The corollary is proved.
\end{proof}
For the next corollary, we recall that a family $\{\varphi_n\}$
of functions in $L^2$ is a \textit{Riesz basis} (see \cite{Gohberg:Krejn}) for $L^2$ if there exist an invertible
operator $A\in\mathfrak{R}(L^2)$ and an orthonormal basis $\{u_n\}$ for $L^2$
such that $\varphi_n=Au_n$ for all $n\in\mathbb{N}$.
\begin{corollary}\label{Riesz}
Let $T\in\mathfrak{R}\left(L^2\right)$ be an integral operator induced
by a $K^\infty$ kernel of Mercer type $\boldsymbol{T}_0$ and suppose that
$T$ is representable as
\begin{equation}\label{merexp}
T=\sum_n \lambda_n\left\langle\cdot,\varphi_n\right\rangle_{L^2}\psi_n,\quad
\{\varphi_n\}\subset \{\varphi^\prime_n\},
\quad \{\psi_n\}\subset \{\psi^\prime_n\},
\end{equation}
where $\{\varphi^\prime_n\}$ and $\{\psi^\prime_n\}$ are Riesz bases for $L^2$
and $\{\lambda_n\}$ is a family of complex numbers. Then
\begin{gather}
\left(D_2^jD_1^i\boldsymbol{T}_0\right)(s,t)=
\sum_n\lambda_n\left[\psi_n\right]^{(i)}(s)
\overline{\left[\varphi_n\right]^{(j)}(t)}, \label{meijT15}
\\
\boldsymbol{t}_0^{(i)}(s)=
\sum_n\overline{\lambda_n\left[\psi_n\right]^{(i)}(s)}\varphi_n,
\quad
\left(\boldsymbol{t}^{\boldsymbol{\prime}}_0\right)^{(j)}(t)=\sum_n\lambda_n
\overline{\left[\varphi_n\right]^{(j)}(t)}\psi_n,\label{meijkt25}
\\
\left[Tf\right]^{(i)}(s)=\sum_n\lambda_n\left\langle f,\varphi_n\right\rangle_{L^2}
\left[\psi_n\right]^{(i)}(s)\label{meijkTf35},
\end{gather}
for all $i$, $j$, and all $s$, $t\in\mathbb{R}$, and all $f\in L^2$, where
the series  of \eqref{meijT15} converges $\mathbb{C}$-absolutely in
$C\left(\mathbb{R}^2,\mathbb{C}\right)$, the two series of \eqref{meijkt25}
both converge in $C\left(\mathbb{R},L^2\right)$, and the series of
\eqref{meijkTf35} converges $\mathbb{C}$-absolutely in
$C(\mathbb{R},\mathbb{C})$.
\end{corollary}
\begin{proof} Assume, with no loss of generality, that
$\lambda_n=\left|\lambda_n\right|e^{\imath\theta_n}\not=0$ for all
$n\in\mathbb{N}$. Since both
$\left\{\varphi^\prime_n\right\}$ and $\left\{\psi^\prime_n\right\}$
are Riesz bases, it follows that there exist two invertible
operators $A$, $B\in\mathfrak{R}(L^2)$ and an orthonormal basis $\{u_n\}$
for $L^2$ such that $T=A\Lambda B$, where
\begin{equation}\label{ALB}
\Lambda=\sum_n \lambda_n\left\langle\cdot,u_{k_n}\right\rangle_{L^2}u_{m_n},
\quad u_{m_n}=A^{-1}\psi_n,\quad u_{k_n}=\left(B^*\right)^{-1}\varphi_n
\ (n\in\mathbb{N})
\end{equation}
for some subsequences $\{m_n\}$, $\{k_n\}$ of the sequence $\{n\}_{n=1}^\infty$.
Introduce operators $U$, $\left|\Lambda\right|
\in\mathfrak{R}\left(L^2\right)$, defined by
\begin{equation*}
U=\sum_n e^{\imath\theta_n}\left\langle\cdot,u_{k_n}\right\rangle_{L^2}  u_{m_n},\quad
\left|\Lambda\right|=\sum_n \left|\lambda_n\right|\left\langle\cdot,u_{k_n}\right\rangle_{L^2}  u_{k_n},
\end{equation*}
and it is clear that  $\Lambda=U\left|\Lambda\right|$ and
$\left|\Lambda\right|\ge0$. If
$W=AU\left|\Lambda\right|^{\frac1{2}}$, $V=B^*\left|\Lambda\right|^{\frac1{2}}$,
then $WW^*=AU\left|\Lambda\right| U^*A^*=TB^{-1}U^*A^*\in\mathcal{M}^{+}(T)$,
$VV^*=B^*\left|\Lambda\right| B=B^*U^*A^{-1}T\in\mathcal{M}^{+}(T)$, and
$T=A\Lambda B=AU\left|\Lambda\right|^{\frac1{2}}\left|\Lambda\right|^{\frac1{2}}B=WV^*$,
that is to say $T=WV^*$ is an $\mathcal{M}$ factorization for $T$.
In addition to this, for each $r\in\mathbb{N}$,
$Wu_r=\delta_{rk_n}\left|\lambda_n\right|^{\frac12}AUu_{k_n}
     =\delta_{rk_n}e^{\imath\theta_n}\left|\lambda_n\right|^{\frac12}Au_{m_n}
     =\delta_{rk_n}e^{\imath\theta_n}\left|\lambda_n\right|^{\frac12}\psi_n$,
and $Vu_r=B^*\left|\Lambda\right|^{\frac12}u_r
    =\delta_{rk_n}\left|\lambda_n\right|^{\frac12}B^*u_{k_n}
    =\delta_{rk_n}\left|\lambda_n\right|^{\frac12}\varphi_n$,
by \eqref{ALB}. Then $\left[Wu_r\right](s)\overline{\left[Vu_r\right](t)}
=\delta_{rk_n}\lambda_n[\psi_n](s)\overline{[\varphi_n](t)}$
for all $s$, $t\in\mathbb{R}$, so that formulae \eqref{meijT15}-\eqref{meijkTf35}
stated in the theorem are already implied by corresponding formulae \eqref{meijT1}-\eqref{meijkTf3}
in Theorem~\ref{mfactor}(b) when the latter is applied, with $\lambda=0$,
to the above $\mathcal{M}$ factorization $T=WV^*$ and basis $\{u_n\}$.
The corollary is proved.
\end{proof}
\begin{remark}
Since, at each fixed regular value $\lambda\in\Pi(T)$, the set
$\left\{R_\lambda(T)\psi^\prime_n\right\}$ remains a Riesz basis for $L^2$ and,
by Theorem~\ref{mfactor}(a), the resolvent kernel
$\boldsymbol{T}_\lambda$ for $\boldsymbol{T}_0$ is a $K^\infty$ kernel of
Mercer type, formulae \eqref{meijT15}-\eqref{meijkTf35} continue to hold with
$0$ and $\psi_n$ replaced respectively by $\lambda$ and $R_\lambda(T)\psi_n$,
due to the same Corollary~\ref{Riesz}.
\end{remark}
\begin{remark}
When the Riesz bases $\{\varphi^\prime_n\}$ and $\{\psi^\prime_n\}$ are both
orthonormal and when $\lambda_n\to0$ as $n\to\infty$, the assumed representation
\eqref{merexp} for $T$ is strongly reminiscent of Schmidt's decomposition
for compact operators (see \cite{Gohberg:Krejn}, and \eqref{JSchdec} below).
For any $T\in\mathfrak{R}(\mathcal{H})$, meanwhile, there is a generalized
Schmidt's decomposition due to I.~Feny\"o \cite{Fenyo}:
\begin{proposition}\label{fenyo}
For every operator $T\in \mathfrak{R}(\mathcal{H})$ there exist an orthonormal 
basis $\{x_n\}\subset\mathrm{Ran\,}T$ for $\overline{\mathrm{Ran\,}T}$,
an orthonormal basis $\{y_n\}\subset \mathrm{Ran\,}T^*$ for
$\overline{\mathrm{Ran\,}T^*}$,
and bounded number sequences $\{\varkappa_n\}$, $\{\mu_n\}$ in $\mathbb{C}$,
such that, for each  $f\in\mathcal{H}$,
\begin{equation}\label{meT1}
Tf=\sum_n\alpha_n\langle f,v_n\rangle_{\mathcal{H}} x_n,\quad
Tf=\sum_{n}\beta_n\langle f,y_n\rangle_{\mathcal{H}} w_n
\end{equation}
in the sense of convergence in $\mathcal{H}$, where
\begin{gather*}
\alpha_n=\sqrt{|\varkappa_n|^2+|\mu_{n-1}|^2},\quad
v_n=\frac1{\alpha_n}\left(\overline \varkappa_ny_n+\overline \mu_{n-1} y_{n-1}\right)
\ (\mu_0=0),
\\
\beta_n=\sqrt{|\varkappa_n|^2+|\mu_n|^2},\quad
w_n=\frac1{\beta_n}\left(\varkappa_nx_n+\mu_nx_{n+1}\right).
\end{gather*}
\end{proposition}
The elements $x_n$, $y_n\in\mathcal{H}$ ($n\in\mathbb{N}$) and the numbers
$\varkappa_n$, $\mu_n\in\mathbb{C}$ ($n\in\mathbb{N}$), whose existence the
proposition guarantees, can be determined simultaneously in a recursive way,
as follows. Let $\varkappa_0=0$, and let $x_0=y_0=0$. If, for some positive
integer $n$, the subset $\{x_j\}_{j=0}^{n-1}$ of $\mathrm{Ran\,}T$, the
element $y_{n-1}\in\mathcal{H}$, and the number $\varkappa_{n-1}\in\mathbb{C}$,
are already defined, and if
\begin{equation}\label{mespanker}
\mathrm{Span}\left(\{x_j\}_{j=0}^{n-1}\right)\not=\overline{\mathrm{Ran\,}T},
\end{equation}
then let the number $\mu_{n-1}$ and the element $x_n$ of $\mathrm{Ran\,}T$ be
chosen to satisfy
\begin{equation}\label{merules1}
\mu_{n-1}x_n=Ty_{n-1}-\varkappa_{n-1}x_{n-1}
\end{equation}
subject to the restrictions
$\sum_{j=0}^{n-1}|\langle x_n,x_j\rangle_{\mathcal{H}}|^2=0$, 
$\|x_n\|_{\mathcal{H}}=1$.
Next, having made the proper choice of $\mu_{n-1}$ and $x_n$, let
$\varkappa_n\in\mathbb{C}$ and $y_n\in\mathcal{H}$ be defined from
$
\overline \varkappa_ny_n=T^*x_n-\overline \mu_{n-1}y_{n-1}
$
provided that $\|y_n\|_{\mathcal{H}}=1$. Either this process of determining
$x_n$, $y_n$, $\mu_{n-1}$, and
$\varkappa_n$, from the previously defined $\{x_j\}_{j=0}^{n-1}$, $y_{n-1}$,
and $\varkappa_{n-1}$, is repeated indefinitely with unending growth of $n$,
or it is terminated whenever inequality \eqref{mespanker} fails to hold. In
any event, the process may generate different sequences $\{x_n\}$, $\{y_n\}$,
$\{\varkappa_n\}$, and $\{\mu_n\}$, because of a freedom in choosing $x_n$
when the right side of \eqref{merules1} happens to be zero. \textit{We do not yet know
whether, given a bi-integral operator $T$ on $\mathcal{H}=L^2$, there always
exist operators $U_1$, $U_2\in\mathfrak{R}(L^2)$ such that $U_1v_n=x_n$ and
$U_2w_n=y_n$ for all $n\in\mathbb{N}$, with the notation of \eqref{meT1}.} If so, then an application of
Theorem~\ref{mfactor}(b) might yield the following result.
\begin{corollary}\label{mefenyo}
If $T\in\mathfrak{R}(L^2)$ is an integral operator with a $K^\infty$ kernel
$\boldsymbol{T}_0$ of Mercer type, then, with the notations of
Proposition~\ref{fenyo},
\begin{gather}
\begin{split}\label{meijT16}
\left(D_2^jD_1^i\boldsymbol{T}_0\right)(s,t)=
\sum_n\alpha_n\left[x_n\right]^{(i)}(s)
\overline{\left[v_n\right]^{(j)}(t)},
\\
\left(D_2^jD_1^i\boldsymbol{T}_0\right)(s,t)=
\sum_n\beta_n\left[w_n\right]^{(i)}(s)
\overline{\left[y_n\right]^{(j)}(t)},
\end{split}
\\
\begin{split}\label{meijkt26}
\boldsymbol{t}_0^{(i)}(s)=
\sum_n\overline{\alpha_n\left[x_n\right]^{(i)}(s)}v_n,
\quad
\left(\boldsymbol{t}^{\boldsymbol{\prime}}_0\right)^{(j)}(t)=\sum_n\alpha_n
\overline{\left[v_n\right]^{(j)}(t)}x_n,
\\
\boldsymbol{t}_0^{(i)}(s)=
\sum_n\overline{\beta_n\left[w_n\right]^{(i)}(s)}y_n,
\quad
\left(\boldsymbol{t}^{\boldsymbol{\prime}}_0\right)^{(j)}(t)=\sum_n\beta_n
\overline{\left[y_n\right]^{(j)}(t)}w_n,
\end{split}
\\
\begin{split}\label{meijkTf36}
\left[Tf\right]^{(i)}(s)=\sum_n\alpha_n\left\langle f,v_n\right\rangle_{L^2}
\left[x_n\right]^{(i)}(s),
\\
\left[Tf\right]^{(i)}(s)=\sum_n\beta_n\left\langle f,y_n\right\rangle_{L^2}
\left[w_n\right]^{(i)}(s),
\end{split}
\end{gather}
for all $i$, $j$, all $s$, $t\in\mathbb{R}$, and all $f\in L^2$, where
the series  of \eqref{meijT16} converge $\mathbb{C}$-absolutely in
$C\left(\mathbb{R}^2,\mathbb{C}\right)$, the series of \eqref{meijkt26}
converge in $C\left(\mathbb{R},L^2\right)$, and the series of
\eqref{meijkTf36} converge $\mathbb{C}$-absolutely in
$C(\mathbb{R},\mathbb{C})$.
\end{corollary}
\end{remark}

\section{Proof of Theorem~\ref{infsmooth}}\label{intrepr}
The proof is broken up into three steps. The first step is to specify a pair
of orthonormal bases, $\left\{f_n\right\}$ for $\mathcal{H}$ and
$\left\{u_n\right\}$ for $L^2$. The second step is to define a unitary
operator from $\mathcal{H}$ onto $L_2$ by sending in a suitable manner the
basis $\{f_n\}$ onto the basis $\{u_n\}$. This operator is suggested as $U$
in the theorem, and the third step of the proof is a straightforward
verification that it is indeed as desired.
\par
\textsl{Step 1.} Let
$\mathcal{S}=\left\{S_\gamma\right\}_{\gamma\in\mathcal{G}}
\subset\mathfrak{R}(\mathcal{H})$ be a family satisfying \eqref{1.2} for some
orthonormal sequence $\left\{e_n\right\}_{n=1}^\infty$ in $\mathcal{H}$. Let
us suppose that we have a pair of orthonormal bases: $\left\{f_n\right\}$ for
$\mathcal{H}$ and $\left\{u_n\right\}$ for $L^2$, where the latter has the
property that, for each $i$, the $i$-th derivative, $[u_n]^{(i)}$, of $[u_n]$
belongs to $C(\mathbb{R},\mathbb{C})$:
\begin{equation}\label{1}
\left[u_n\right]^{(i)}\in C(\mathbb{R},\mathbb{C})\quad(n\in\mathbb{N}).
\end{equation}
Let us suppose further that each of these bases can be subdivided into two
infinite subsequences: $\left\{f_n\right\}$ into $\{x_k\}_{k=1}^\infty$ and
$\{y_k\}_{k=1}^\infty$, while $\left\{u_n\right\}$ into $\{g_k\}_{k=1}^\infty$
and $\{h_k\}_{k=1}^\infty$, such that
\begin{equation}\label{sumdxk}
\sum_k d(x_k)\le 2
\end{equation}
and, for each $i$,
\begin{gather}
\label{hki}
\sum_k \left\|[h_k]^{(i)}\right\|_{C(\mathbb{R},\mathbb{C})}<\infty,
\\ \label{zndn}
\sum_k d(x_k)\left\|[g_k]^{(i)}\right\|_{C(\mathbb{R},\mathbb{C})}<\infty,
\end{gather}
where
\begin{equation}\label{dxk}
d(x_k):=2\left(\sup_{\gamma\in\mathcal{G}}\left\|S_\gamma x_k\right\|_{\mathcal{H}}^{1/4}+
\sup_{\gamma\in\mathcal{G}}\left\|\left(S_\gamma\right)^*x_k\right\|_{\mathcal{H}}^{1/4}\right)
\quad(k\in\mathbb{N}).
\end{equation}
The proof will make use of the orthonormal bases $\left\{f_n\right\}$,
$\left\{u_n\right\}$ just described to define the action of the desired
unitary operator $U\colon\mathcal{H}\to L^2$ in the theorem.
By using the assumption \eqref{1.2} and some facts from the
theory of wavelets, the following example affirmatively answers the existence
question for such a pair of orthonormal bases.
\begin{example}
Let $\mathbb{N}$ be decomposed into two infinite subsequences
$\left\{l(k)\right\}_{k=1}^\infty$ and $\left\{m(k)\right\}_{k=1}^\infty$, and
let $\{u_n\}$ be an orthonormal basis for $L^2$ that has property \eqref{1},
and such that, for each $i$,
\begin{equation}\label{2}
\left\|\left[u_n\right]^{(i)}\right\|_{C(\mathbb{R},\mathbb{C})}\le N_nD_i
\quad(n\in\mathbb{N}),
\end{equation}
where $\{N_n\}_{n=1}^\infty$, $\{D_i\}_{i=0}^\infty$ are two sequences
of positive reals, the first of which is subject to the restriction that
\begin{equation}\label{3}
\sum_kN_{l(k)}<\infty.
\end{equation}
\par
Then the orthonormal basis $\{u_n\}$ can be paired (in the sense above) with
an orthonormal basis $\{f_n\}$ for $\mathcal{H}$. Indeed, from conditions
\eqref{2}, \eqref{3} it follows that the requirement \eqref{hki} may be
satisfied for all $i$ by taking $h_k=u_{l(k)}$ ($k\in\mathbb{N}$). At the same
time, on account of \eqref{1.2} and of \eqref{2}, the sequence
$\left\{e_k\right\}_{k=1}^\infty$ does always have an infinite subsequence
$\left\{x_k\right\}_{k=1}^\infty$ that satisfies both the requirement
\eqref{sumdxk} and the requirement \eqref{zndn} for each $i$  with
$g_k=u_{m(k)}$ ($k\in\mathbb{N}$). Once such a subsequence
$\left\{x_k\right\}_{k=1}^\infty$ of $\left\{e_k\right\}_{k=1}^\infty$ has
been fixed, the remaining task is simply to complete the $x_k$'s to an
orthonormal basis, and to let $y_k$ ($k\in\mathbb{N}$) denote the new elements
of that basis. Now $\{f_n\}=\{x_k\}_{k=1}^\infty\cup\{y_k\}_{k=1}^\infty$ and
$\{u_n\}=\{g_k\}_{k=1}^\infty\cup\{h_k\}_{k=1}^\infty$ constitute a pair of
orthonormal bases satisfying \eqref{1}-\eqref{dxk}.
\par
In turn, an explicit example of a basis $\{u_n\}$ that obeys the above three
conditions \eqref{1}, \eqref{2}, and \eqref{3}, can be adopted from the
wavelet theory, as follows. Let $\psi$ be the Lemari\'e-Meyer wavelet,
\begin{equation*}\label{wave}
[\psi](s)=\dfrac1{2\pi}\int_{\mathbb{R}}e^{\imath\xi(\frac12+s)}
\mathrm{sgn\,}\xi b(\left|\xi\right|)\,d\xi\quad (s\in\mathbb{R})
\end{equation*}
with the bell function $b$ being infinitely smooth and compactly supported on
$[0,+\infty)$ (see \cite[Example~D, p.~62]{Her:Wei} for details). Then
$[\psi]$ is of the Schwartz class $\mathcal{S}(\mathbb{R})$, so its every
derivative $\left[\psi\right]^{(i)}$ is in $C(\mathbb{R},\mathbb{C})$.
In addition, the ``mother wavelet'' $\psi$ generates an orthonormal basis
$\left\{\psi_{\alpha\beta}\right\}_{\alpha,\,\beta\in\mathbb{Z}}$ for $L^2$ by
\begin{equation*}\label{ujk}
\psi_{\alpha\beta}=2^{\frac \alpha2}\psi(2^\alpha\cdot-\beta)\quad  (\alpha,\,\beta\in\mathbb{Z}).
\end{equation*}
\par
In any manner, rearrange the two-indexed set
$\{\psi_{\alpha\beta}\}_{\alpha,\,\beta\in\mathbb{Z}}$ into a simple sequence,
so that it looks like $\left\{u_n\right\}_{n=1}^\infty$; clearly each term
here, $u_n$, has property \eqref{1} for each $i$. Besides, it is easily
verified that a norm estimate like \eqref{2} holds for each $[u_n]^{(i)}$:
the factors in the right-hand side of \eqref{2} may be taken as
\begin{equation*}
N_n=\begin{cases}
2^{\alpha_n^2}&\text{if $\alpha_n>0$,}\\
2^{\alpha_n/2}&\text{if $\alpha_n\le0$,}
\end{cases}
\qquad D_i=2^{\left(i+1/2\right)^2}\left\|[\psi]^{(i)}\right\|_{C(\mathbb{R},\mathbb{C})},
\end{equation*}
with the convention that $u_n=\psi_{\alpha_n\beta_n}$ ($n\in\mathbb{N}$)
in conformity with that rearrangement. This choice of $N_n$ also gives
$\sum_kN_{l(k)}<\infty$ (cf. \eqref{3}) whenever
$\left\{l(k)\right\}_{k=1}^\infty$ is a subsequence of
$\left\{l\right\}_{l=1}^\infty$ satisfying $\alpha_{l(k)}\to -\infty$ as
$k\to\infty$.
\end{example}
\par
\textsl{Step 2}.
In this step our intention is to construct a candidate for the desired unitary 
operator $U\colon\mathcal{H}\to L^2$ in the theorem. Recalling that
\begin{equation}
\begin{gathered}\label{SET}
\{f_1,f_2,f_3,\dots\}=\{x_1,x_2,x_3,\dots\}\cup\left\{y_1,y_2,y_3,\dots\right\},
\\
\{u_1,u_2,u_3,\dots\}=\{g_1,g_2,g_3,\dots\}\cup\left\{h_1,h_2,h_3,\dots\right\},
\end{gathered}
\end{equation}
where the unions are disjoint, define the unitary operator $U$ from
$\mathcal{H}$ onto $L^2$ by setting
\begin{equation}\label{2.12}
Ux_k=g_k,\quad Uy_k=h_k
\quad\text{for all $k\in \mathbb{N}$,}
\end{equation}
in the harmless assumption that $Uf_n=u_n$ for all $n\in\mathbb{N}$.
\par
\textsl{Step 3.} This step of the proof is to show that, in fact, if the
unitary operator $U\colon \mathcal{H}\to L^2$ is defined as in \eqref{2.12},
then the operators $US_{\gamma}U^{-1}\colon L^2\to L^2$
($\gamma\in\mathcal{G}$) are all simultaneously bi-Carleman operators with
$K^\infty$ kernels of Mercer type.
\par
Fix, to begin with, an arbitrary index $\gamma\in\mathcal{G}$, abbreviate
$S_\gamma$ to $S$, and let $T=USU^{-1}$. It is to be proved that $T$ is an
integral operator with a $K^\infty$ kernel of Mercer type. The idea is first
to conveniently split the operators $S$, $S^*$, each into two parts as follows.
If $E$ is the orthogonal projection of $\mathcal{H}$ onto
$\mathrm{Span}\left(\left\{x_k\right\}_{k=1}^\infty\right)$, write
\begin{equation}\label{2.7}
S=(I_\mathcal{H}-E)S+ES,\quad S^*=(I_\mathcal{H}-E)S^*+ES^*.
\end{equation}
It is an immediate consequence of \eqref{sumdxk} and \eqref{SET} that the
operators $J=S E$ and $\widetilde{J}=S^* E$ are Hilbert-Schmidt. Write down
their Schmidt's decompositions
\begin{equation}\label{JSchdec}
J=\sum_n s_{n}\left\langle\cdot,p_{n}\right\rangle_{\mathcal{H}}q_{n},
\quad\widetilde{J}=\sum_n \widetilde{s}_{n}\left\langle\cdot,\widetilde{p}_{n}
\right\rangle_{\mathcal{H}}\widetilde{q}_{n},
\end{equation}
and introduce compact operators $B$ and $\widetilde{B}$, which are defined by
\begin{equation}\label{aux}
B=\sum_n s_{n}^{\frac1{4}}
\left\langle\cdot,p_{n}\right\rangle_{\mathcal{H}} q_{n},
\quad
\widetilde{B}=\sum_n \widetilde{s}_{n}^{\frac1{4}}\left\langle\cdot,
\widetilde{p}_{n}\right\rangle_{\mathcal{H}}\widetilde{q}_{n};
\end{equation}
here the $s_{n}$ are the singular values of $J$ (eigenvalues of
$\left(J J^*\right)^{\frac1{2}}$), $\left\{p_n\right\}$, $\left\{q_n\right\}$
are orthonormal sets in $\mathcal{H}$ (the $p_{n}$ are eigenvectors for $J^*J$
and $q_{n}$ are eigenvectors for $JJ^*$). The explanation of the notation for
$\widetilde{J}$ is similar.
\par
For each $f\in\mathcal{H}$, let
\begin{equation}\label{2.8n}
c(f):=
\left\|Bf\right\|_{\mathcal{H}}
+
\left\|B^*f\right\|_{\mathcal{H}}
+
\left\|\widetilde{B}f\right\|_{\mathcal{H}}
+
\left\|\left(\widetilde{B}\right)^*f\right\|_{\mathcal{H}}. 
\end{equation}
Applying Schwarz's inequality yields
\allowdisplaybreaks
\begin{equation}\label{2.8}
\begin{split}
c(x_k)&=
\sqrt{\sum_ns_n^{\frac12}
\left|\left\langle x_k,p_n\right\rangle_{\mathcal{H}}\right|^2}
+
\sqrt{\sum_ns_n^{\frac12}
\left|\left\langle x_k,q_n\right\rangle_{\mathcal{H}}\right|^2}
\\&\qquad
+
\sqrt{\sum_n\widetilde{s}_n^{\frac12}
\left|\left\langle x_k,\widetilde{p}_n\right\rangle_{\mathcal{H}}\right|^2}
+
\sqrt{\sum_n\widetilde{s}_n^{\frac12}
\left|\left\langle x_k,\widetilde{q}_n\right\rangle_{\mathcal{H}}\right|^2}
\\&
=
\left\|\left(J^*J\right)^{\frac18}x_k\right\|_{\mathcal{H}}
+
\left\|\left(JJ^*\right)^{\frac18}x_k\right\|_{\mathcal{H}}
\\&\qquad
+
\left\|\left(\left(\widetilde{J}\right)^*\widetilde{J}\right)^{\frac18}x_k\right\|_{\mathcal{H}}
+
\left\|\left(\widetilde{J}\left(\widetilde{J}\right)^*\right)^{\frac18}x_k\right\|_{\mathcal{H}}
\\&
\leq
\left\|J x_k\right\|_{\mathcal{H}}^{\frac14}
+
\left\|J^*x_k\right\|_{\mathcal{H}}^{\frac14}
+
\left\|\widetilde{J}x_k\right\|_{\mathcal{H}}^{\frac14}
+
\left\|\left(\widetilde{J}\right)^*x_k\right\|_{\mathcal{H}}^{\frac14}
\\&\qquad
\le 2\left(\left\|Sx_k\right\|^{\frac14}+
\left\|S^*x_k\right\|^{\frac14}\right)\le d(x_k)\quad(k\in\mathbb{N}),
\end{split}
\end{equation}
whence it follows again from \eqref{sumdxk} and \eqref{SET} that $B$ and
$\widetilde{B}$ are both Hilbert-Schmidt operators, implying
\begin{equation}\label{2.9}
\sum_ns_n^{\frac1{2}}<\infty,\quad
\sum_n\widetilde{s}_n^{\frac1{2}}<\infty.
\end{equation}
\par
Now define $Q=(I_\mathcal{H}-E)S^*$, $\widetilde{Q}=(I_\mathcal{H}-E)S$. Since $\{y_k\}$ is, by
construction, an orthonormal basis for the subspace $(I_\mathcal{H}-E)\mathcal{H}$, it
follows that
\begin{equation}
\begin{gathered}\label{2.10}
Qf=\sum_k\left\langle Qf,y_k\right\rangle_{\mathcal{H}}y_k
=
\sum_k\left\langle f,Sy_k\right\rangle_{\mathcal{H}}y_k,\\
\widetilde{Q}f=\sum_k\left\langle\widetilde{Q}f,y_k\right\rangle_{\mathcal{H}}y_k
=
\sum_k\left\langle f,S^* y_k\right\rangle_{\mathcal{H}}y_k
\end{gathered}
\end{equation}
for all $f\in\mathcal{H}$, with the series converging in $\mathcal{H}$.
\par
Having considered the splittings \eqref{2.7}, which now look like
\begin{equation}\label{2.7new}
S=\widetilde{Q}+\left(\widetilde{J}\right)^*,\quad S^*=Q+J^*,
\end{equation}
our next task is to verify that each of the four operators $UQU^{-1}$,
$U\widetilde QU^{-1}$, $UJ^*U^{-1}$, and $U\left(\widetilde J\right)^*U^{-1}$,
is a Carleman operator with a kernel satisfying conditions (i), (ii) in
Definition~\ref{Kmkernel}. The checking is straightforward, and goes by
representing all pertinent kernels and Carleman functions as infinitely smooth
sums of termwise differentiable series of infinitely smooth functions, without
any use of the method of Section~\ref{prmfactor}. Theorems on termwise
differentiation of series will however be used repeatedly (and usually tacitly)
in the subsequent analysis.
\par
If $P=UQU^{-1}$, $\widetilde{P}=U\widetilde{Q}U^{-1}$, then from \eqref{2.10}
and \eqref{2.12} it follows that, for each $f\in L^2$,
\begin{equation}\label{P}
Pf=\sum_k \left\langle f,Th_k\right\rangle_{L^2} h_k,
\quad
\widetilde{P}f=\sum_k\left\langle f,T^*h_k\right\rangle_{L^2}h_k,
\end{equation}
in the sense of convergence in $L^2$. Represent the equivalence classes
$Th_k$, $T^*h_k\in L^2$ ($k\in\mathbb{N}$) by the Fourier expansions with
respect to the orthonormal basis $\{u_n\}$:
\begin{equation*}
Th_k=\sum_n \left\langle y_k,S^*f_n\right\rangle_{\mathcal{H}}u_n,
\quad
T^*h_k=\sum_n \left\langle y_k,Sf_n\right\rangle_{\mathcal{H}} u_n,
\end{equation*}
where the convergence is in the $L^2$ norm. But somewhat more than that can be
said about convergence, namely that, for each fixed $i$, the series
\begin{equation}\label{2.15}
\sum_n \left\langle y_k,S^*f_n\right\rangle_{\mathcal{H}}[u_n]^{(i)}(s),
\quad
\sum_n \left\langle y_k,Sf_n\right\rangle_{\mathcal{H}}[u_n]^{(i)}(s)\quad(k\in\mathbb{N})
\end{equation}
converge in the norm of $C(\mathbb{R},\mathbb{C})$. Indeed, all the series are
dominated everywhere on $\mathbb{R}$ by one series
\begin{equation*}
\sum_n
\left(\left\|S^*f_n\right\|_{\mathcal{H}}+
\left\|Sf_n\right\|_{\mathcal{H}}\right)\left|[u_n]^{(i)}(s)\right|,
\end{equation*}
which is uniformly convergent on $\mathbb{R}$ for the following reason: its
component subseries (cf. \eqref{2.12})
\begin{gather*}
\sum_k\left(\left\|Sx_k\right\|_{\mathcal{H}}+\left\|S^*x_k\right\|_{\mathcal{H}}\right)
\left|[g_k]^{(i)}(s)\right|,\quad
\sum_k\left(\left\|Sy_k\right\|_{\mathcal{H}}+\left\|S^*y_k\right\|_{\mathcal{H}}\right)
\left|[h_k]^{(i)}(s)\right|
\end{gather*}
are uniformly convergent on $\mathbb{R}$ because they are in turn dominated by
the convergent series
\begin{equation}\label{2.16}
\sum_k d(x_k)\left\|[g_k]^{(i)}\right\|_{C(\mathbb{R},\mathbb{C})},
\quad
\sum_k 2\|S\|\left\|[h_k]^{(i)}\right\|_{C(\mathbb{R},\mathbb{C})},
\end{equation}
respectively (see \eqref{dxk}, \eqref{sumdxk}, \eqref{zndn}, and \eqref{hki}).
\par
It is now evident that the pointwise sums of the series of \eqref{2.15} define
functions that belong to $C(\mathbb{R},\mathbb{C})$ and that are none other
than $\left[Th_k\right]^{(i)}$, $\left[T^*h_k\right]^{(i)}$ ($k\in\mathbb{N}$),
respectively. Moreover, the above arguments prove that given any $i$, there
exists a positive constant $C_i$ such that
\begin{equation*}
\left\|\left[Th_k\right]^{(i)}\right\|_{C(\mathbb{R},\mathbb{C})}<C_i,
\quad
\left\|\left[T^*h_k\right]^{(i)}\right\|_{C(\mathbb{R},\mathbb{C})}<C_i,
\end{equation*}
for all $k$. Hence, by \eqref{hki}, it is possible to infer that, for all $i$,
$j$, the series
\begin{equation*}
\sum_k \left[h_k\right]^{(i)}(s)
\overline{\left[Th_k\right]^{(j)}(t)},\quad
\sum_k \left[h_k\right]^{(i)}(s)
\overline{\left[T^*h_k\right]^{(j)}(t)}
\end{equation*}
converge and even absolutely in the norm of
$C\left(\mathbb{R}^2,\mathbb{C}\right)$. This makes it clear that functions
$\boldsymbol{P}$, $\boldsymbol{\widetilde{P}}\colon\mathbb{R}^2\to\mathbb{C}$,
defined by
\begin{equation}\label{2.14}
\boldsymbol{P}(s,t)=\sum_k \left[h_k\right](s)\overline
{\left[Th_k\right](t)},\quad
\boldsymbol{\widetilde{P}}(s,t)=\sum_k\left[h_k\right](s)
\overline{\left[T^*h_k\right](t)},
\end{equation}
satisfy condition (i) in Definition~\ref{Kmkernel}.
\par
Now we prove that (Carleman) functions $\boldsymbol{p}$,
$\boldsymbol{\widetilde{p}}\colon\mathbb{R}\to L^2$, defined by
\begin{equation}\label{pp}
\boldsymbol{p}(s)=\overline{\boldsymbol{P}(s,\cdot)}=\sum_k
\overline{\left[h_k\right](s)}Th_k,\quad
\boldsymbol{\widetilde{p}}(s)=\overline{\boldsymbol{\widetilde{P}}(s,\cdot)}=
\sum_k\overline{\left[h_k\right](s)}T^*h_k,
\end{equation}
are both subject to requirement (ii) in Definition~\ref{Kmkernel}. Indeed, the
series displayed converge absolutely in the $C\left(\mathbb{R},L^2\right)$
norm, because those two series whose terms are respectively
$\left|\left[h_k\right](s)\right|\|Th_k\|_{L^2}$ ($k\in\mathbb{N}$) and
$\left|\left[h_k\right](s)\right|\left\|T^*h_k\right\|_{L^2}$ ($k\in\mathbb{N}$)
are dominated by the second series of \eqref{2.16} with $i=0$. For the
remaining $i$, a similar reasoning implies the same convergence behavior of
the series
$\sum_k\overline{\left[h_k\right]^{(i)}(s)}Th_k$,
$\sum_k\overline{\left[h_k\right]^{(i)}(s)}T^*h_k.$
The asserted property of both Carleman functions $\boldsymbol{p}$ and
$\boldsymbol{\widetilde{p}}$ to satisfy (ii) then follows by the corresponding
theorem on termwise differentiation of series.
\par
From \eqref{hki} it follows that the series of \eqref{P}, viewed as series in
$C(\mathbb{R},\mathbb{C})$, converge (and even absolutely) in
$C(\mathbb{R},\mathbb{C})$ norm, and therefore that their pointwise sums are
none other than $\left[Pf\right]$ and $\left[\widetilde{P}f\right]$,
respectively. On the other hand, the established properties of the series of
\eqref{pp} and of \eqref{2.14} make it possible to write, for each temporarily
fixed $s\in\mathbb{R}$, the following chains of relations
\allowdisplaybreaks
\begin{equation*}
\begin{split}
\sum_k&\left\langle f,Th_{k}\right\rangle_{L^2} \left[h_{k}\right](s)
=\left\langle f,\sum\limits_k\overline{\left[h_{k}\right](s)}Th_{k}\right\rangle_{L^2}
\\&
=\int_\mathbb{R} \left(\sum\limits_k\left[h_{k}\right](s)
         \overline{\left[Th_{k}\right](t)}\right)f(t)\,dt
         =\int_{\mathbb{R}}\boldsymbol{P}(s,t)f(t)\,dt,
\end{split}
\end{equation*}
\begin{equation*}
\begin{split}
\sum_k& \left\langle f,T^*h_{k}\right\rangle_{L^2} \left[h_{k}\right](s)
=\left\langle f,\sum\limits_k\overline{\left[h_{k}\right](s)}T^*h_{k}\right\rangle_{L^2}
\\&
=\int_\mathbb{R} \left(\sum\limits_k\left[h_{k}\right](s)
         \overline{\left[T^*h_{k}\right](t)}\right)f(t)\,dt=\int_{\mathbb{R}}\boldsymbol{\widetilde{P}}(s,t)f(t)\,dt
\end{split}
\end{equation*}
whenever $f$ is in $L^2$. These imply that $P$ and $\widetilde{P}$ are
Carleman integral operators, the kernels of which are $\boldsymbol{P}$ and
$\boldsymbol{\widetilde{P}}$ respectively, both are subject to requirements (i),
(ii) in Definition~\ref{Kmkernel} by what precedes.
\par
Now we consider the Hilbert-Schmidt integral operators $F=UJ^*U^{-1}$ and
$\widetilde{F}=U\left(\widetilde{J}\right)^*U^{-1}$ on $L^2$, and prove that
the kernels of these operators are $K^\infty$ kernels. Associate with Schmidt's
decompositions \eqref{JSchdec} for $J$, $\widetilde{J}$ two functions
$\boldsymbol{F}$, $\boldsymbol{\widetilde{F}}\colon\mathbb{R}^2\to\mathbb{C}$,
defined by
\begin{equation}\label{2.17}
\begin{split}
\boldsymbol{F}(s,t)
=\sum_n s_n^{\frac1{2}}\left[U B^*q_n\right] (s)
&\overline{\left[U Bp_n\right](t)}
\\&\left(=\sum_n s_n \left[U p_n\right](s)\overline{\left[U q_n\right](t)}\right),
\\
\boldsymbol{\widetilde{F}}(s,t)
=\sum_n\widetilde{s}_n^{\frac1{2}}\left[U\left(\widetilde{B}\right)^*\widetilde{q}_n
\right](s)&\overline{\left[U\widetilde{B}\widetilde{p}_n\right](t)}
\\&\left(=\sum_n\widetilde{s}_n\left[U\widetilde{p}_n\right](s)
\overline{\left[U\widetilde{q}_n\right](t)}\right),
\end{split}
\end{equation}
whenever $s$, $t$ are in $\mathbb{R}$; for the auxiliary operators $B$,
$\widetilde{B}$ here used, see \eqref{aux}. It is to be noted that without the
square brackets the bilinear series just written do converge almost everywhere
on $\mathbb{R}^2$ to Hilbert-Schmidt kernels that induce respectively $F$ and
$\widetilde{F}$ (see \eqref{2.9}). Hence, and again because of \eqref{2.9},
the conclusion that the above-defined functions $\boldsymbol{F}$
and $\boldsymbol{\widetilde{F}}$ are the kernels of $F$ and $\widetilde{F}$,
respectively, and are subject to condition (i) of Definition~\ref{Kmkernel}
can be inferred as soon as it is known that for each fixed $i$ the terms of
the sequences
\begin{equation}\label{jsequences}
\left\{\left[U Bp_k\right]^{(i)}\right\},
\quad
\left\{\left[U B^*q_k\right]^{(i)}\right\},
\quad
\left\{\left[U\widetilde{B}\widetilde{p}_k\right]^{(i)}\right\},
\quad
\left\{\left[U\left(\widetilde{B}\right)^*\widetilde{q}_k\right]^{(i)}\right\}
\end{equation}
make sense, are all in $C(\mathbb{R},\mathbb{C})$, and their
$C(\mathbb{R},\mathbb{C})$ norms are bounded, regardless of $k$. To see that
the conditions just listed are all fulfilled, it suffices to observe that,
once $i$ is fixed, all the series
\begin{equation}\label{Fseries}
\begin{gathered}
\sum_n\left\langle p_k,B^*f_n\right\rangle_{\mathcal{H}}\left[u_n\right]^{(i)}(s),\quad
\sum_n\left\langle q_k,Bf_n\right\rangle_{\mathcal{H}} \left[u_n\right]^{(i)}(s),\quad
\\\sum_n\left\langle\widetilde{p}_k,\left(\widetilde{B}\right)^*f_n\right\rangle_{\mathcal{H}} \left[u_n\right]^{(i)}(s),\quad
\sum_n\left\langle\widetilde{q}_k,\widetilde{B}f_n\right\rangle_{\mathcal{H}}
\left[u_n\right]^{(i)}(s)
\quad(k\in\mathbb{N})
\end{gathered}
\end{equation}
(which, with $i=0$, are merely the  $L^2$-convergent Fourier expansions, with
respect to the orthonormal basis $\{u_n\}$, for $UBp_k$, $UB^*q_k$,
$U \widetilde{B}\widetilde{p}_k$, and $U(\widetilde{B})^*\widetilde{q}_k$,
respectively) are dominated by one series
\begin{equation*}
\sum_nc(f_n)\left|\left[u_n\right]^{(i)}(s)\right|,
\end{equation*}
where $c(f)$ is defined in \eqref{2.8n} above. This last series is uniformly
convergent on $\mathbb{R}$, because it is composed of the two subseries
\begin{equation*}
\sum_nc(x_k)\left|\left[g_k\right]^{(i)}(s)\right|,
\quad
\sum_nc(y_k)\left|\left[h_k\right]^{(i)}(s)\right|
\end{equation*}
that converge uniformly in $\mathbb{R}$, having as their dominant series the
convergent series
\begin{equation*}
\sum_kd(x_k)\left\|[g_k]^{(i)}\right\|_{C(\mathbb{R},\mathbb{C})},
\quad
\sum_k2\left(\|B\|+\left\|\widetilde{B}\right\|\right)\left\|[h_k]^{(i)}\right\|_{C(\mathbb{R},\mathbb{C})}
\end{equation*}
(see \eqref{2.8}, \eqref{zndn}, \eqref{hki}). Thus, for each $i$, all the
series in \eqref{Fseries} converge (and even absolutely) in the
$C(\mathbb{R},\mathbb{C})$ norm, and their sums are none other than,
respectively, $[UBp_k]^{(i)}$, $\left[UB^*q_k\right]^{(i)}$,
$\left[U\widetilde{B}\widetilde{p}_k\right]^{(i)}$, and
$\left[U(\widetilde{B})^*\widetilde{q}_k\right]^{(i)}$, ($k\in\mathbb{N}$).
Then, in virtue of \eqref{2.9} and the above-established boundedness of the
sequences of \eqref{jsequences} in $C(\mathbb{R},\mathbb{C})$, the series
\allowdisplaybreaks
\begin{equation*}
\begin{gathered}
\sum_ns_n^{\frac12}\left[UB^*q_n\right]^{(i)}(s)\overline{\left[UBp_n\right]^{(j)}(t)},
\quad
\sum_n\widetilde{s}_n^{\frac12}\left[U\left(\widetilde{B}\right)^*
\widetilde{q}_n\right]^{(i)}(s)\overline{\left[U\widetilde{B}\widetilde p_n\right]^{(j)}(t)}
\end{gathered}
\end{equation*}
converge (and even absolutely) in $C\left(\mathbb{R}^2,\mathbb{C}\right)$, for all $i$,
$j$. This in conjunction with \eqref{2.17} is sufficient to conclude that both
the functions $\boldsymbol{F}$ and $\boldsymbol{\widetilde{F}}$ satisfy
condition (i) of Definition~\ref{Kmkernel}, and are the Hilbert-Schmidt kernels
of $F$ and of $\widetilde{F}$, respectively.
\par
Again by the properties of the sequences of \eqref{jsequences} and by
\eqref{2.9}, the series
\begin{gather*}\label{}
\sum_n s_n^{\frac1{2}}\overline{\left[U B^*q_n\right]^{(i)}(s)}U Bp_n,
\quad
\sum_n s_n^{\frac12}UB^*q_n\overline{\left[UBp_n\right]^{(i)}(s)},
\\
\sum_n\widetilde{s}_n^{\frac1{2}}\overline{\left[U\left(\widetilde{B}\right)^*\widetilde{q}_n
\right]^{(i)}(s)}U \widetilde{B}\widetilde{p}_n,
\quad
\sum_n \widetilde{s}_n^{\frac12}U\left(\widetilde{B}\right)^*\widetilde{q}_n
\overline{\left[U\widetilde{B}\widetilde{p}_n\right]^{(i)}(s)}
\end{gather*}
converge and even absolutely in the $C\left(\mathbb{R},L^2\right)$ norm, for
each $i$. Observe, via \eqref{2.17}, that four of these series, namely those
with $i=0$, represent the Carleman functions
$\boldsymbol{f}(s)=\overline{\boldsymbol{F}(s,\cdot)}$,
$\boldsymbol{f}^{\boldsymbol{\prime}}(s)=\boldsymbol{F}(\cdot,s)$,
$\boldsymbol{\widetilde{f}}(s)=\overline{\boldsymbol{\widetilde{F}}(s,\cdot)}$,
and $\boldsymbol{\widetilde{f}}^{\boldsymbol{\prime}}(s)=
\boldsymbol{\widetilde{F}}(\cdot,s)$, respectively, which therefore do
satisfy conditions (ii), (iii) in Definition~\ref{Kmkernel}.
This finally implies that the Hilbert-Schmidt kernels $\boldsymbol{F}$ and
$\boldsymbol{\widetilde{F}}$ are $K^\infty$ kernels of $F$ and of $\widetilde{F}$,
respectively.
\par
In accordance with \eqref{2.7new}, the operator $T$ (which, recall, is the
transform by $U$ of $S$) and its adjoint decompose as
$T=\widetilde{P}+\widetilde{F}$, $T^*=P+F$ where all the terms are already
known to be Carleman operators, with kernels satisfying (i), (ii). Hence, both
$T$ and $T^*$ are Carleman operators, and their kernels, say $\boldsymbol{T}$
and $\boldsymbol{T}^{\boldsymbol\ast}$, defined as
\begin{equation}\label{2.18}
\boldsymbol{T}(s,t)=\boldsymbol{\widetilde{P}}(s,t)+\boldsymbol{\widetilde{F}}(s,t),
\quad
\boldsymbol{T}^{\boldsymbol{\ast}}(s,t)=\boldsymbol{P}(s,t)+\boldsymbol{F}(s,t)
\end{equation}
for all $s$, $t\in\mathbb{R}$, inherit the smoothness properties (i), (ii)
from their terms. But then, since the operator $T$ is bi-integral,
$\boldsymbol{T}(s,t)=\overline{\boldsymbol{T}^{\boldsymbol{\ast}}(t,s)}$
for all $s$, $t\in\mathbb{R}$; hence $\boldsymbol{T}(\cdot,t)=
\overline{\boldsymbol{T}^{\boldsymbol{\ast}}(t,\cdot)}$ in the $L^2$ sense
for each fixed $t\in\mathbb{R}$. That implies that the kernel $\boldsymbol{T}$
also satisfies condition (iii) in Definition~\ref{Kmkernel}, and is thus a
$K^\infty$ kernel that induces the bi-Carleman operator $T$.
\par
This $K^\infty$ kernel $\boldsymbol{T}$ is moreover of Mercer type. A main
ingredient in proving the claim is the following remark:
if $d(x_k)$ are defined as in \eqref{dxk}, if $U$ is defined as in
\eqref{2.12}, and if an operator $A\in\mathfrak{R}(\mathcal{H})$ fulfils
\begin{equation*}
2\left(\left\|Ax_k\right\|_{\mathcal{H}}^{\frac1{4}}
+\left\|A^*x_k\right\|_{\mathcal{H}}^{\frac1{4}}\right)\leq
d(x_k)\quad\text{for each $k\in\mathbb{N}$,}
\end{equation*}
then, like $T=USU^{-1}$, the operator $UAU^{-1}$ is a bi-Carleman operator
with a $K^\infty$ kernel, no matter whether $A$ is in the initial family
$\mathcal{S}$ or not. This additional feature of $U$ may be checked by directly
applying the above verification procedure, leading from \eqref{2.7} to
\eqref{2.18}, to the operator $A$ in place of $S$.
\par
If now a nonzero operator $A$ is restricted to lie in the set $\mathcal{M}(S)$,
then whatever its representation from among those in \eqref{relations} is, the
scalar multiple $A_1=\left(16\max\left\{\|M\|,\|N\|\right\}\right)^{-1}\cdot A$
of $A$ obeys the following easily verifiable inequality valid for all
$k\in\mathbb{N}$:
\begin{equation*}
2\left(\left\|A_1x_k\right\|_{\mathcal{H}}^{\frac1{4}}+
\left\|A_1^*x_k\right\|_{\mathcal{H}}^{\frac1{4}}\right)\leq
2\left(\left\|Sx_k\right\|_{\mathcal{H}}^{\frac1{4}}+
\left\|S^*x_k\right\|_{\mathcal{H}}^{\frac1{4}}\right)\leq
d(x_k).
\end{equation*}
By the above remark, this implies that the unitary operator $U$ defined in
\eqref{2.12} carries, besides $S$, every other member, $A$, of $\mathcal{M}(S)$
onto an integral operator, $UAU^{-1}$, with a $K^\infty$ kernel.
Then, since $\mathcal{M}(T)=U\mathcal{M}(S)U^{-1}$, the $K^\infty$ kernel
$\boldsymbol{T}$ of $T$, constructed in \eqref{2.18} above, is of Mercer type
by virtue of the definition. Moreover, since $T=US_\gamma U^{-1}$ where the
index $\gamma$ was arbitrarily fixed, it follows that all the operators
$T_\gamma=US_\gamma U^{-1}$ ($\gamma\in\mathcal{G}$) are bi-Carleman operators
whose kernels are $K^\infty$ kernels of Mercer type.
\par
The only thing that remains to be proved is that those $K^\infty$ kernels which
induce finite linear combinations of $US_\gamma U^{-1}$ ($\gamma\in\mathcal{G}$)
are also of Mercer type. Indeed, consider any finite linear combination
$G=\sum z_\gamma S_\gamma$ with $\sum\left|z_\gamma\right|\leq 1$. It is seen
easily that, for each $n$,
\begin{equation*}
\left\|\sum z_\gamma S_\gamma e_n\right\|_{\mathcal{H}}
\leq\sup_{\gamma\in\mathcal{G}}\left\|S_\gamma e_n\right\|_{\mathcal{H}},
\quad
\left\|\sum\overline{z}_\gamma \left(S_\gamma\right)^* e_n\right\|_{\mathcal{H}}
\leq\sup_{\gamma\in\mathcal{G}}\left\|\left(S_\gamma\right)^*e_n\right\|_{\mathcal{H}}.
\end{equation*}
There is, therefore, no barrier to assuming that $G$ was, from the start, in
the initial family $\mathcal{S}$. The proof of the theorem is now complete.

\end{document}